\documentclass[12pt,a4paper,reqno]{amsart}
\usepackage{amssymb,latexsym,amsfonts,amsthm,upref,amsmath,capt-of}
\usepackage[english]{babel}
\usepackage[margin=1in]{geometry} 
\usepackage[foot]{amsaddr}
\usepackage[dvipsnames,x11names]{xcolor}
\usepackage{hyperref}   
\hypersetup{colorlinks,citecolor=green,filecolor=black,linkcolor=red,urlcolor=blue}
\usepackage{enumerate} 
\usepackage{tikz}
\usetikzlibrary{chains}
\usepackage{float}
\usepackage{cleveref}
\usepackage{mathtools}
\usepackage{cite}
\usepackage{filecontents}

\makeatletter

\newcommand{\leqnomode}{\tagsleft@true}
\newcommand{\reqnomode}{\tagsleft@false}

\makeatother

\newtheorem{thm}{Theorem}

\newtheorem{lem}[thm]{Lemma}
\numberwithin{equation}{section}
\numberwithin{thm}{section}
\newtheorem{cor}[thm]{Corollary}
\newtheorem{ex}[thm]{Example}

\usepackage{mathtools}
\usepackage{tikz}
\usetikzlibrary{chains}
\usepackage{graphics}
\usepackage{epic,eepic}

\newcommand{\beq}{\begin{equation}}
\newcommand{\eeq}{\end{equation}}
\newcommand{\beqn}{\begin{equation*}}
\newcommand{\eeqn}{\end{equation*}}

\newcommand{\g}{\mathfrak{g}}
\newcommand{\wtg}{\widetilde{\mathfrak{g}}}
\newcommand{\whg}{\widehat{\mathfrak{g}}}
\newcommand{\h}{\mathfrak{h}}
\newcommand{\whh}{\widehat{\mathfrak{h}}}
\newcommand{\wth}{\widetilde{\mathfrak{h}}}
\newcommand{\n}{\mathfrak{n}}
\newcommand{\wtn}{\widetilde{\mathfrak{n}}}
\newcommand{\whq}{\widehat{Q}}
\newcommand{\whn}{\widehat{\nu}}
\newcommand{\Rc}{\mathcal{R}}

\newcommand{\Ec}{\mathcal{E}}
\newcommand{\en}{\mathop{\mathrm{en}}}
\newcommand{\chg}{\mathop{\mathrm{chg}}}
\newcommand{\ts}{\hspace{1pt}}
\newcommand{\gtgeq}{\widehat{\mathfrak{g}}[\widehat{\nu}]^{\geq 0}}
\newcommand{\gtle}{\widehat{\mathfrak{g}}[\widehat{\nu}]^{< 0}}

\newcommand{\ntplusleq}{\widetilde{\mathfrak{n}}_+^{< 0}}
\newcommand{\Pc}{\mathcal{P}}
\newcommand{\Zc}{\mathcal{Z}}

\makeatletter
\def\smalloverbrace#1{\mathop{\vbox{\m@th\ialign{##\crcr\noalign{\kern3\p@}%
  \tiny\downbracefill\crcr\noalign{\kern3\p@\nointerlineskip}%
  $\hfil\displaystyle{#1}\hfil$\crcr}}}\limits}
\makeatother

\makeatletter
\def\smallunderbrace#1{\mathop{\vtop{\m@th\ialign{##\crcr
   $\hfil\displaystyle{#1}\hfil$\crcr
   \noalign{\kern3\p@\nointerlineskip}%
   \tiny\upbracefill\crcr\noalign{\kern3\p@}}}}\limits}
\makeatother

\begin{document}

\title{Combinatorial bases of standard modules of twisted affine Lie algebras in types $A_{2l-1}^{(2)}$ and $D_{l+1}^{(2)}$: rectangular highest    weights}

\author{Marijana Butorac}
\address[M. Butorac]{Faculty of Mathematics, University of Rijeka, Radmile Matej\v{c}i\'{c} 2, 51000 Rijeka, Croatia}
\email{mbutorac@math.uniri.hr}

\author{Slaven Ko\v{z}i\'{c}} 
\address[S. Ko\v{z}i\'{c}]{Department of Mathematics, Faculty of Science, University of Zagreb, Bijeni\v{c}ka cesta 30, 10000 Zagreb, Croatia}
\email{kslaven@math.hr}



\begin{abstract}
We consider the standard modules of rectangular highest weights of   affine Lie algebras in types $A_{2l-1}^{(2)}$ and $D_{l+1}^{(2)}$.
By using   vertex algebraic techniques   we construct the  combinatorial bases for   standard modules and their principal subspaces and parafermionic spaces.  Finally,  we compute the corresponding character formulae and, as an application, we obtain two new families of  combinatorial identities.
\end{abstract}

\maketitle

\allowdisplaybreaks


\numberwithin{equation}{section}

\section{Introduction}
The connection between   Rogers--Ramanujan-type sums and   characters of parafermi\-o\-nic spaces was originally established in the work  of J. Lepowsky and M. Primc  \cite{LP}   by constructing the vertex operator bases of parafermi\-o\-nic spaces for the affine Lie algebra   $A_1^{(1)}$. Later on, such combinatorial bases for the parafermionic spaces of standard modules of rectangular highest weights  for all untwisted affine Lie algebras  were found by G. Georgiev \cite{G2} (type $A_l^{(1)}$) and  by  M. Primc and the authors \cite{BKP} (types $B_l^{(1)}$--$G_2^{(1)}$). These bases were used to prove the generalized  Rogers--Ramanujan-type sums obtained by A. Kuniba, T. Nakanishi and J. Suzuki  \cite{KNS} and D. Gepner \cite{Gep1}. 
Generalizing this     approach, M. Okado and R. Takenaka \cite{OT,T} found bases for the parafermionic spaces of standard modules of  highest weights $k\Lambda_0$   for   twisted affine Lie algebras. Moreover, using these bases they proved character formulae   conjectured  by G. Hatayama, A. Kuniba, M. Okado, T. Takagi and Z. Tsuboi 
 \cite{HKOTT}. 

In this paper, we construct the combinatorial bases for the parafermionic  spaces of standard modules of  {\em rectangular} highest weights, i.e  of weights of the form   
\beq\label{weight}
\Lambda=k_0\Lambda_0+k_j\Lambda_j,\qquad\text{where}\qquad k_0, k_j \in \mathbb{Z}_{\geq 0}\quad\text{and}\quad k=k_0+k_j > 0 ,
\eeq
 for   twisted affine Lie algebras   of types $A_{2l-1}^{(2)}$ and $D_{l+1}^{(2)}$, where   $j=1$ (resp. $j=l$) in type $A_{2l-1}^{(2)}$ (resp. $D_{l+1}^{(2)}$).
As with \cite{BKP,G2,OT}, the construction consists of several steps, as follows.

 First, we obtain the   quasi-particle bases  of the principal subspaces $W_{L^{\whn}(\Lambda)}$ of standard modules $L^{\whn}(\Lambda)$ for the highest weight $\Lambda$ as in \eqref{weight}. As with the corresponding bases in the case   $\Lambda=k\Lambda_0$,   found by the first author and C. Sadowski    \cite{BuS}, they are expressed in terms of monomials of twisted quasi-particles applied on the highest weight vector of $L^{\whn}(\Lambda)$, such that the quasi-particle energies satisfy certain difference and initial conditions. However, in this case,
the quasi-particle monomials satisfy different
  initial conditions, which we determine by employing   results of J. Lepowsky \cite{L1} and B. Bakalov and V. Kac \cite{BKac} on the lattice construction of level one twisted modules. 
Despite an apparent analogy between   the cases $\Lambda=k\Lambda_0$ and $\Lambda=k_0\Lambda_0+k_j\Lambda_j$, the latter requires a certain new ingredient in the proof of linear independence. More specifically, we   use   the so-called simple current operator from   Li's papers \cite{Li} and \cite{Li1}, which is a certain linear bijection between level one twisted modules from \cite{L1}. In addition, we also construct quasi-particle bases for the principal subspaces of   generalized Verma modules.

Next, we use the quasi-particle bases for  principal subspaces to construct the combinatorial bases for   standard modules of rectangular highest weights in types $A_{2l-1}^{(2)}$ and $D_{l+1}^{(2)}$, thus generalizing the aforementioned results of M. Okado and R. Takanaka \cite{OT}. As before, the proof of  linear independence relies on Li's simple current operators. Finally, we employ the parafermionic projection to   obtain the combinatorial bases of parafermionic spaces for all rectangular highest weights in types $A_{2l-1}^{(2)}$ and $D_{l+1}^{(2)}$.

At the end, we use all these bases to compute the character formulae of the corresponding standard modules and their principal subspaces and parafermionic spaces. In particular,   we obtain the parafermionic character formulae which coincide  with those of D. Gepner \cite{Gep2}. Furthermore, by comparing the quasi-particle bases for the principal subspaces of the generalized Verma modules with their Poincar\' e--Birkhoff--Witt bases, we discover  two new families of combinatorial identities which correspond to types $A_{2l-1}^{(2)}$ and $D_{l+1}^{(2)}$. The identities of such form often appear in different areas of mathematics; see, e.g., the papers by J. Fulman \cite{Fu} and J. Hua \cite{Hu}.

\section{Preliminaries}
\subsection{Level one standard modules}
Let $\g$ be a complex simple Lie algebra of type  $A_{2l-1}$ or $D_{l+1}$  with the Cartan subalgebra $\h$ and the root system $R$,  equipped with the standard nondegenerate symmetric invariant bilinear form $\langle \cdot, \cdot \rangle$. Using this form one can identify $\h$ and its dual $\h^{\ast}$ via $\langle  \alpha, h \rangle=\alpha(h)$ for $\alpha \in \h^{\ast}$ and $h \in \h$. We assume that the form is rescaled so that $\langle  \alpha, \alpha \rangle=2$  when $\alpha$ is a simple root.  The simple roots $ \alpha_1, \ldots , \alpha_{\text{rk} \g}$, where $\text{rk} \g$ stands for the rank of $\g$, generate the root lattice $Q$ and the fundamental weights $\lambda_1, \ldots, \lambda_{\text{rk} \g}$ generate the weight lattice $P$ of $\g$. Let
$\nu$ be the Dynkin diagram automorphism of $Q$ of order 2 as in Figure \ref{figure}.

\tikzset{node distance=2em, ch/.style={circle,draw,on chain,inner sep=2pt},chj/.style={ch,join},every path/.style={shorten >=4pt,shorten <=4pt},line width=1pt,baseline=-1ex}

\newcommand{\alabel}[1]{%
{\tiny \(\mathrlap{#1}\)}
}

\newcommand{\mlabel}[1]{%
  \(#1\)
}

\let\dlabel=\alabel
\let\ulabel=\mlabel

\newcommand{\dnode}[2][chj]{%
\node[#1,label={below:\dlabel{#2}}] {};
}

\newcommand{\dnodea}[3][chj]{%
\dnode[#1,label={above:\ulabel{#2}}]{#3}
}

\newcommand{\dnodear}[1]{%
\dnode[chj,label={above right:\dlabel{#1}}] {}
}

\newcommand{\dnodeanj}[2]{%
\dnodea[ch]{#1}{#2}
}

\newcommand{\dnodenj}[1]{%
\dnode[ch]{#1}
}

\newcommand{\dnodebr}[1]{%
\node[chj,label={below right:\dlabel{#1}}] {};
}

\newcommand{\dnoder}[2][chj]{%
\node[#1,label={right:\dlabel{#2}}] {};
}

\newcommand{\dydots}{%
\node[chj,draw=none,inner sep=1pt] {\dots};
}

\newcommand{\QRightarrow}{%
\begingroup
\tikzset{every path/.style={}}%
\tikz \draw (0,3pt) -- ++(1em,0) (0,1pt) -- ++(1em+1pt,0) (0,-1pt) -- ++(1em+1pt,0) (0,-3pt) -- ++(1em,0) (1em-1pt,5pt) to[out=-75,in=135] (1em+2pt,0) to[out=-135,in=75] (1em-1pt,-5pt);
\endgroup
}

\newcommand{\QLeftarrow}{%
\begingroup
\tikz
\draw[shorten >=0pt,shorten <=0pt] (0,3pt) -- ++(-1em,0) (0,1pt) -- ++(-1em-1pt,0) (0,-1pt) -- ++(-1em-1pt,0) (0,-3pt) -- ++(-1em,0) (-1em+1pt,5pt) to[out=-105,in=45] (-1em-2pt,0) to[out=-45,in=105] (-1em+1pt,-5pt);
\endgroup
}
\begin{align*}
&A_{2l-1}   &&\hspace{-15pt}
\begin{tikzpicture}[start chain]
\dnode{1}
\dnode{2}
\dydots
\dnode{l-1}
\dnode{l}
\dnode{l+1}
\dydots
\dnode{2l-2}
\dnode{2l-1}
\draw[latex-latex,shorten >=2mm,shorten <=2mm] (chain-1)  to[bend left]     (chain-9);
\draw[latex-latex,shorten >=2mm,shorten <=2mm] (chain-2)  to[bend left]     (chain-8);
\draw[latex-latex,shorten >=2mm,shorten <=2mm] (chain-4)  to[bend left]    (chain-6);
\end{tikzpicture}\\
%
&D_{l+1} &&\hspace{-15pt}
\begin{tikzpicture}
\begin{scope}[start chain]
\dnode{1}
\dnode{2}
\node[chj,draw=none] {\dots};
\dnode{l-1}
\dnode{l}
\end{scope}
\begin{scope}[start chain=br going above]
\chainin(chain-4);
\dnodear{l+1}
 \draw[latex-latex,shorten >=2mm,shorten <=2mm] (3.7,1)  to[bend left]     (chain-5);
 \end{scope}
\end{tikzpicture}
\end{align*}
\begingroup\vspace*{-\baselineskip}

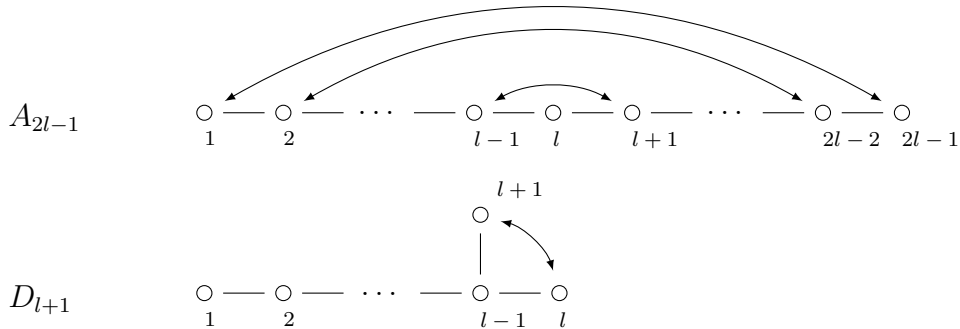
\captionof{figure}{Finite Dynkin diagrams and the action of the automorphism $\nu$}\label{figure}
\vspace*{\baselineskip}\endgroup

Let us recall the construction of standard modules of twisted affine Lie algebra $\wtg[\whn]$ associated to $\g$ and $\nu$ from \cite{L1} (see also \cite{BKac,CalLM4,DL1,FLM1,FLM2,Li}). The central extension $\whq$ of the root lattice $Q$ by the cyclic group $\left<-1\right>=\left\{\pm 1\right\}$,  
\beq\label{ce1}
1\longrightarrow\left<-1\right>\longrightarrow\whq\overset{\overline{~~}}{\longrightarrow} Q\longrightarrow 1
\eeq
is determined by the commutator map
$ C_0(\alpha, \beta)=(-1)^{\langle \alpha, \beta \rangle}$,   
i.e.  by
$aba^{-1}b^{-1}=C_0(\overline{a},\overline{b})$ 
for $a,b \in \whq$. 
We fix a section $e:Q \rightarrow \whq$ normalized so that   $e_0=1$ and $\overline{e_{\alpha}}=\alpha$ for all $\alpha \in Q$   and  
$$
 e_{\alpha}e_{\beta}=\epsilon_{C_0}(\alpha,\beta)e_{\alpha+\beta}~ \text{ for all }~\alpha,\beta\in Q.
$$
Here $\epsilon_{C_0}(\cdot, \cdot)$ denotes the normalized 2-cocycle $\epsilon_{C_0}:Q\times Q \rightarrow \langle-1\rangle$ defined as in \cite{PS2}.

Let $\whn$ be the lifting of $\nu$ to the involutive automorphism of $\whq$ such that  
\beqn
  \overline{\whn a}=\nu \overline{a}  \quad \text{for all}\quad  a \in \whq\qquad\text{and} \qquad  \whn a=a \quad \text{ if }  \quad  \nu\overline{a}=\overline{a}.
   \end{equation*}
By the same symbol we denote the extension of the automorphism $\whn$ to the lattice vertex operator algebra $V_Q$.   
Moreover, for any $k$ we  write $\whn=\whn^{\otimes k}$ for the extension of   $\whn$   to the vertex operator algebra $V_Q^{\otimes k}$. Again, we have $\whn^2=1$ (cf. \cite{BKac,L1}). 
 
For any integer $m  $ set 
\beqn
\h_{(m)} = \{ x \in \h \  | \ \nu(x) =(-1)^m x \}.
\eeqn
This gives us the decomposition of the Cartan subalgebra
\begin{equation*}
\h = \h_{(0)} \oplus \h_{(1)},\qquad\text{where}\qquad 
\h_{(0)}=\bigoplus_{i=1}^l\mathbb{C}\left( \alpha_i + \nu\alpha_i\right)
\quad\text{and}\quad
\h_{(1)}=\bigoplus_{i=1}^l\mathbb{C}\left( \alpha_i - \nu\alpha_i\right).
\end{equation*}
For any $\alpha \in \h$ and $m\in\mathbb{Z}$ we write $\alpha_{(m)}$ for the projection of $\alpha$ onto $\h_{(m)}$. 

Let
\beqn
\whh[\nu] ^{\pm}= \mathfrak{h}_{(0)} \otimes  t^{\pm 1}\mathbb{C}[t^{\pm 1}] \oplus \h_{(1)} \otimes t^{\pm 1/2}\mathbb{C}[t^{\pm 1}].
\eeqn 
Consider the Heisenberg subalgebra 
\beqn
\whh[\nu] _{\frac{1}{2}\mathbb{Z}}= \whh[\nu] ^{+} \oplus \whh[\nu] ^{-}\oplus \mathbb{C}c  
\eeqn  
of the twisted affine Lie algebra
\beqn
\wth[\nu] = \coprod_{m \in \mathbb{Z}} \h_{(m)} \otimes t^{m/2} \oplus \mathbb{C}c\oplus \mathbb{C}d= \h_{(0)} \otimes  \mathbb{C}[t,t^{-1}] \oplus \h_{(1)} \otimes t^{1/2} \mathbb{C}[t,t^{-1}] \oplus \mathbb{C}c\oplus \mathbb{C}d.
\eeqn
Furthermore, let $S[\nu]$ be the irreducible $\whh[\nu]_{\frac{1}{2}\mathbb{Z}}$-module
\beqn
S[\nu]=U  ( \wth[\nu]  )\otimes_{U_+}\mathbb{C},\quad\text{where}\quad
U_+=U(\textstyle\prod_{m\in\frac{1}{2}\mathbb{Z}_{\geq 0}}\mathfrak{h}_{(2m)}\otimes t^m\oplus \mathbb{C}c\oplus \mathbb{C}d).
\eeqn
Note that $S[\nu]$ is linearly isomorphic to the   space of   symmetric algebra $S (\whh[\nu] ^{-} )$. 

Form the induced $\whq $-module  
\begin{equation*}
U_{T}=\mathbb{C}[\whq ]\otimes_{\mathbb{C}[\widehat{N}]}T\cong \mathbb{C}[Q/N],\end{equation*}
where   $\widehat{N}$ is the subgroup of $\whq $ obtained by pulling back $N=\h_{(1)}\cap Q$, and $T$ denotes the one dimensional $\widehat{N}$-module $\mathbb{C}$ with character $\tau: \widehat{N} \rightarrow \mathbb{C}^{\times}$. 
Note that $U_T$ can be regarded 
as an $\wth[\nu]$-module with
\beqn \label{opd}
d \cdot a=- \textstyle \frac{1}{2}\langle \overline{a}+\nu \overline{a},\overline{a}\rangle a.
\eeqn
Moreover, the $\wth[\nu]$-module $U_T$ is isomorphic to $\mathbb{C}[\pi_0 Q] \otimes T$, where $\pi_0$ denotes the projection from $\h$ onto $\h_{(0)}$ 
 (cf. \cite{L1,CalLM4}).

Let $V_{Q}^{T}$ be the irreducible $\whn$-twisted module for $V_Q$ given by
\begin{equation*}
V_{Q}^{T}=S[\nu]\otimes U_T.\end{equation*}
For each $e_{\alpha}\in\whq$, we consider $\whn$-twisted vertex operator
\beqn
Y^{\whn}(\iota(e_{\alpha}),z)=2^{-\frac{\left<\alpha,\alpha\right>}{2}}E^{-}(-\alpha,z)E^{+}(-\alpha,z)e_{\alpha}z^{\alpha_{(0)}+\frac{\left<\alpha_{(0)},\alpha_{(0)}\right>}{2}-\frac{\left<\alpha,\alpha\right>}{2}},\eeqn
 where
\beqn
E^{\pm}(-\alpha,z)=\text{exp}\left(\sum_{m\in\pm\frac{1}{2}\mathbb{Z}_{+}}\frac{-\alpha_{(2m)}(m)}{m}z^{-m}\right).
\eeqn
The component operators $x_{\alpha}^{\whn}(m)$  are defined by 
\beqn
Y^{\whn}(\iota(e_{\alpha}),z)=\sum_{m\in\frac{1}{2}\mathbb{Z}}x_{\alpha}^{\whn}(m)z^{-m-\frac{\left<\alpha,\alpha\right>}{2}}.
\eeqn
By \cite{L1}, 
they satisfy 
\beq\label{VertexOperators18}
x_{\nu\alpha}^{\whn}(m)= (-1)^{2m}x_{\alpha}^{\whn}(m) \quad\text{for all}\quad m \in\textstyle \frac{1}{2}\mathbb{Z}.
\eeq

We shall use the symbol $\whn$ to denote the
  lifting of the corresponding automorphism   on $\h$ to an automorphism of $\g$, which is given by
\beqn
\whn x_{ \alpha}= \begin{cases}
\epsilon_{C_0}(\alpha,\alpha) x_{ \nu \alpha}, &\text{if } \g\text{ is of type }A_{2l-1},\\
x_{ \nu \alpha}, &\text{if } \g\text{ is of type }D_{l+1}.
\end{cases}
\eeqn
For any $m\in\mathbb{Z}$ set
\beqn
\g_{(m)}=\{x\in\g\  | \ \whn(x)=(-1)^m x\} .
\eeqn
We shall denote by $x_{(m)}$ the projection of the element $x\in\g$ onto $\g_{(m)}$.
Let $\wtg[\whn]$ be the $\whn$-twisted affine Lie algebra  associated with $\g$ and $\whn$,
\beqn 
\wtg[\whn]=\whg[ \whn]\oplus\mathbb{C}d,\qquad\text{where}\qquad \whg[ \whn]=\coprod_{m\in\frac{1}{2}\mathbb{Z}}\mathfrak{g}_{(2m)}\otimes t^{m} \oplus\mathbb{C}c .
\eeqn
Its Lie  brackets are given by
\beqn
[x\otimes t^m,y\otimes t^n]=[x,y]\otimes t^{m+n}+\left<x,y\right>m\delta_{m+n,0}c, \quad 
[c,\wtg[ \whn]]=0,  \quad [d,x\otimes t^n]=n x\otimes t^n ,
\eeqn
for $x\in\mathfrak{g}_{(2m)}$, $y\in\mathfrak{g}_{(2n)}$ and $m,n\in\frac{1}{2}\mathbb{Z}$.
The Lie algebra $\wtg[ \whn]$ is isomorphic to the twisted affine Lie algebra of type $A_{2l-1}^{(2)}$ or $D_{l+1}^{(2)}$ with respect to the choice of $Q$ and $\nu$ (cf. \cite{K}).

The representation of $\wth[\nu]$ on $V_Q^T$ uniquely extends  to the representation of $\wtg[\whn]$   by 
\begin{equation*}
(x_{\alpha})_{(2m)}\otimes t^m\mapsto  x^{\hat{\nu}}_{\alpha}(m)\quad\text{for}\quad m\in\textstyle\frac{1}{2}\mathbb{Z},\, \alpha\in Q.
\end{equation*}
 Moreover, we have $V_Q^T \cong L^{\whn}(\Lambda_0)$, where $L^{\whn}(\Lambda_0)$  is the level one standard $\wtg[\whn]$-module of the highest weight $\Lambda_0$ with the highest weight vector $v_{  \Lambda_0 }=v_{L^{\whn}(\Lambda_0)}=1  \otimes 1$ (cf. \cite{L1}). 
 Let 
\beq \label{gama1}
\gamma=\pi_0 \lambda_j,
\eeq
where, as in \eqref{weight}, we set $j=1$ (resp. $j=l$) for $\g$ of type $A_{2l-1}$ (resp. $D_{l+1}$).  
The element  $\gamma \in \h_{(0)}$ satisfies 
$\left< \gamma , \alpha \right> \in \frac{1}{2}Q$  for all $\alpha \in Q$.
Following \cite{L1}  (see also \cite{CalLM4})  we obtain the level one standard  $\wtg[\whn]$-module $L^{\whn}(\Lambda_j)$  of the highest weight $\Lambda_j$  by defining a new representation of $\wth[\nu]$ on $V_Q^T$ as follows.  
The elements of $\h_{(0)}$ act on $U_T$ by
\beqn
\alpha^{\gamma} \cdot b \otimes t=\left< \alpha, \gamma +\overline{b}\right> b \otimes t
\eeqn
for $b \in \whq$ , $t \in T$ , $\alpha \in \h_{(0)}$, the operators $z^{\alpha^{\gamma}}$ act by 
\beqn
z^{\alpha^{\gamma}}\cdot b \otimes t= z^{\left<\alpha, \gamma + \overline{b} \right>}b \otimes t  
\eeqn
and  the Heisenberg algebra $\whh[\nu] _{\frac{1}{2}\mathbb{Z}}$ acts trivially. A $\gamma$-shifted $\whn$-twisted vertex operator $Y^{\whn, \gamma}(\iota(e_{\alpha}),z)$ is defined by
$$
Y^{\whn, \gamma}(\iota(e_{\alpha}),z)=
Y^{\whn}(\iota(e_{\alpha}),z)z^{\left<\alpha_{(0)}, \gamma\right>}
 = 2^{-\frac{\left<\alpha,\alpha\right>}{2}}E^{-}(-\alpha,z)E^{+}(-\alpha,z)e_{\alpha}z^{\alpha_{(0)}^{\gamma}+\frac{\left<\alpha_{(0)},\alpha_{(0)}\right>}{2}-\frac{\left<\alpha,\alpha\right>}{2}} 
$$
(cf. \cite{CalLM4}). The component operators $x_{\alpha}^{\whn, \gamma}(m)$, $m \in \frac{1}{2}\mathbb{Z}$, are given by  
\beqn 
Y^{\whn, \gamma}(\iota(e_{\alpha}),z)=\sum_{m\in\frac{1}{2}\mathbb{Z}}x_{\alpha}^{\whn, \gamma}(m)z^{-m-1}=x_{\alpha}^{\whn, \gamma}(z),
\eeqn
and they satisfy
\beq\label{svo}
x_{\alpha}^{\whn, \gamma}(m) =x_{\alpha}^{\whn}(m+\left<\alpha_{(0)}, \gamma\right>).
\eeq
 
Let $\whn_{\gamma}$ be  the automorphism of   Lie algebra  $\g$ such that
$$
\whn_{\gamma} h  = \nu h  \quad\text{for all} \quad h \in \h  
\qquad\text{and}\qquad
\whn_{\gamma} x_{\alpha}  = (-1)^{2\left<\alpha, \gamma\right>}\whn x_{\alpha}  \quad\text{for all} \quad\alpha \in Q.
$$ 
Then  $\whn_{\gamma}^2=1$  and the $\whn_{\gamma}$-twisted affine Lie algebra 
$\wtg[\whn_{\gamma}]$
associated to $\g$   coincides with $\wtg[\whn]$. Furthermore, from   \cite[Thm. 10.1]{L1} (see also   \cite[Thm. 4.2]{CalLM4}) follows that the new representation of $\wth[\nu]$ on $V_Q^{T,\gamma}$ uniquely extends  to the representation of $\wtg[\whn]$  by 
\begin{equation*}
(x_{\alpha})_{(2m)}\otimes t^m\mapsto  x^{\whn,\gamma}_{\alpha}(m)\quad\text{for all}\quad  m\in\textstyle\frac{1}{2}\mathbb{Z},\,\alpha\in Q. 
\end{equation*}
 It can be easily checked that we have the isomorphism of $\wtg[\whn]$-modules $V_Q^{T,\gamma}\cong L^{\whn}(\Lambda_j)$. We   denote the highest weight vector of $L^{\whn}(\Lambda_j)$ by $v_{ \Lambda_j}$. By    \cite[Prop. 5.4]{Li}, $ V_Q^{T,\gamma} $ is a $ \nu_{\gamma}\whn$-twisted $V_Q$-module, where 
$\nu_{\gamma}=e^{-2\pi \sqrt{-1}\gamma (0)}$
is an automorphism of $V_Q$.

\subsection{Higher level standard modules}
Let $k \in \mathbb{N}$ be fixed. The tensor product $(V_Q^{T, \gamma  })^{\otimes k}$ of $k$ copies of $\whn_{\gamma}$-twisted $V_Q$-modules $V_Q^{T, \gamma}$   
possesses a structure of  $\whn_{\gamma}$-twisted module for the vertex operator algebra $V_Q^{\otimes k}$ with vertex operators
\beqn 
Y^{\whn, \gamma}( v_1\otimes \cdots \otimes v_k, z)=Y^{\whn,\gamma}( v_1, z)\otimes \cdots \otimes Y^{\whn,\gamma}( v_k, z).\eeqn
The operators $x_{\alpha}^{\whn,\gamma} (m)$ on $(V_Q^{T, \gamma  })^{\otimes k}$ are given by
$$
Y^{\whn,\gamma}(x_{\alpha}(-1)\cdot({\bf 1}\otimes \cdots \otimes {\bf 1}), z)=\sum_{m \in \frac{1}{2}\mathbb{Z}}x_{\alpha}^{\whn, \gamma} (m)z^{-m-1}.
$$
 This construction gives the   diagonal action of the sublattice $kQ\subset Q$ on $(V_Q^{T, \gamma  })^{\otimes k}$:
\begin{align}
k\alpha \mapsto \rho (k\alpha)=e^{\alpha}\otimes \ldots \otimes e^{\alpha},\quad \alpha\in Q.\label{ro}
\end{align}

To simplify the notation, in the rest of the paper we will omit the symbol $\gamma$ and write $ x^{\whn} (z)$ for  $ x^{\whn, \gamma} (z)$ as well, as it will be clear from the context whether  $\gamma$ is   $0$ or $\pi_0\lambda_j$.
We shall need the following relations on the standard $\wtg[\whn]$-module $L^{\whn}(\Lambda)$, which are consequence of the adjoint action of $e_{\alpha}$ on $\wtg$ (cf. \cite{K}, see also \cite{OT,P}):
\begin{eqnarray}\label{e1}
e_{\alpha} d e_{\alpha}^{-1}&=&d + \alpha -\textstyle\frac{\langle\alpha_{(0)}, \alpha_{(0)}\rangle}{2}c,\\
\label{e2}
e_{\alpha} c e_{\alpha}^{-1}&=&c,\\
\label{e3}
e_{\alpha} h_{(0)} e_{\alpha}^{-1}&=&h_{(0)}-\langle\alpha_{(0)},h_{(0)}\rangle c, \ h \in \h, \\
\label{e4}
e_{\alpha} h_{(2s)}(s) e_{\alpha}^{-1}&=&h_{(2s)}(s), \ s \neq 0, \\
\label{e5}
e_{\alpha} x^{\whn}_{\beta}(s) e_{\alpha}^{-1}&=&x^{\whn}_{\beta}(s-\langle\beta_{(0)},\alpha_{(0)}\rangle).
\end{eqnarray}

For any positive integer $r  $ consider the twisted vertex operators 
\beq\label{kv1}
x_{r\alpha_i}^{\whn}(z)=Y^{\whn}(x_{\alpha_i}(-1)^r{\bf 1}, z)=\sum_{m \in \frac{1}{2}\mathbb{Z}}x_{r\alpha_i}^{\whn}(m)z^{-m-r}=(Y^{\whn}(x_{\alpha_i}(-1){\bf{1}},z))^r,
\eeq
associated with the vector $x_{\alpha_i}(-1)^r{\bf{1}} \in L(k\Lambda_0)$ (cf. \cite{Li}). They satisfy the commutator formula for twisted vertex operators,
\beqn\label{comform}
\left[ x_{\alpha}^{\whn}(z_1),  x_{r\beta}^{\whn}(z_2) \right]=
\sum_{s \geq 0} 
\frac{(-1)^s}{2s!}
\left(\frac{d}{dz_1}\right)^s z_2^{-1}  \sum_{q \in \mathbb{Z}/2\mathbb{Z}}\delta\left((-1)^{q}\frac{z_1^{\frac{1}{2}}}{z_2^{\frac{1}{2}}}\right)Y^{\whn}(\nu^q x_{\alpha}(s)x_{\beta}(-1)^{r}{\bf 1}, z_2).
\eeqn
For any $h  \in \h$ and $m \in \frac{1}{2}\mathbb{Z}$ we have
\beq\label{comform1}
\left[h_{(2m)}(m) ,  x_{r\alpha}^{\whn}(z ) \right]= r\langle h_{(2m)}, \alpha_{(-2m)}\rangle z^m x_{r\alpha}^{\whn}(z ).
\eeq

For any dominant integral weight $\lambda \in (\h_{(0)})^{\ast}$  let $U_{\lambda}$ be the finite-dimensional irreducible
$\g_{(0)}$-module of highest weight $\lambda$.  The generalized Verma module $N^{\whn}(\Lambda)$ associated with $U_{\lambda}$ is defined by 
$N^{\whn}(\Lambda)= U(\whg[\whn])\otimes_{U(\gtgeq)} U_{\lambda}$, 
where the action of the Lie algebra
$$\gtgeq=\bigoplus_{n \in \frac{1}{2}\mathbb{Z}_{\geq 0} } (\g_{(2n)} \otimes t^{n}) \oplus \mathbb{C} c$$
on $U_{\lambda}$ is given by
$$\g_{(2n)}\otimes t^n \cdot u=0\quad \text{for all}\quad n> 0 \qquad\text{and}\qquad c\cdot u=k u \quad\text{for all}\quad u\in U_{\lambda}.$$ 
Recall that the Poincar\'{e}--Birkhoff--Witt theorem for the universal enveloping algebra gives the vector space isomorphism
$$N^{\whn}(\Lambda) \cong U(\gtle)\otimes_{\mathbb{C}} U_{\lambda} ,\quad\text{where}\quad \gtle=\bigoplus_{n \in \frac{1}{2}\mathbb{Z}_{< 0} } (\g_{(2n)} \otimes t^{n}).$$

Let $V$ denote  the generalized Verma module $N^{\whn}(\Lambda)$ or its unique simple quotient $L^{\whn}(\Lambda)$. Then there exists a linear map 
\begin{align*}
Y^{\whn}: \g(-1)\otimes {\bf 1} &\rightarrow   (\text{End}V) [[z^{\frac{1}{2}}, z^{-\frac{1}{2}}]],\\
\nonumber
Y^{\whn} (x(-1)\otimes {\bf 1}, z) &=x^{\whn}(z)= \sum_{n \in \frac{1}{2}\mathbb{Z}} x^{\whn}(n)z^{-n-1} ,
\end{align*}
 which defines a structure of $\whn$-twisted $N(k\Lambda_0)$-module over $V$; see, e.g., \cite{Li}. Its coefficients   $x^{\whn}(n)$  satisfy the identity in \eqref{VertexOperators18} as well.

\subsection{Principal subspaces}
Consider 
 the nilpotent Lie subalgebra
\beqn
\mathfrak{n}=\bigoplus_{\alpha \in R_+}\mathbb{C}x_{\alpha}\eeqn
 of $\g$  generated by all root vectors $x_{\alpha}$ which  correspond to positive roots $\alpha \in R_+$. Let 
 \beqn
\wtn[\whn]=\overline{\n}[\whn] \oplus \mathbb{C}c ,\qquad\text{where}\qquad
\overline{\n}[\whn] =\coprod_{m \in \frac{1}{2}\mathbb{Z},\,\alpha\in R_+}\mathbb{C}{x_{\alpha}}_{(2m)} \otimes t^m,
\eeqn
be
the corresponding subalgebra of $\wtg[\whn]$.
Denote by $v_{\Lambda}$   the highest weight vector of   $V=N^{\whn}(\Lambda),L^{\whn}(\Lambda)$.
Following \cite{FS}, define  the
{\em  principal subspace}  of $V=N^{\whn}(\Lambda), L^{\whn}(\Lambda)$ by
\beqn
W_{V}^T= U(\overline{\n}[\whn])v_{\Lambda}.
\eeqn   

The principal subspace $W_{V}^T$ coincides with
\beqn
  U(\overline{\n}_{\alpha_l}[\whn])U(\overline{\n}_{\alpha_{l-1}}[\whn])  \cdots U(\overline{\n}_{\alpha_1}[\whn])v_{\Lambda}, \quad \text{where} \quad \overline{\n}_{\alpha_i}[\whn]=\coprod_{m\in\frac{1}{2}\mathbb{Z}}\mathbb{C}{x_{\alpha_i}}_{(2m)}\otimes t^m,
\eeqn
  (see   \cite[Lemma 3.1]{BuS} and \cite[Lemma 3.1]{G}). In addition, as with the untwisted case in \cite[Sect. 5]{BK}, for $V=N^{\whn}(\Lambda)$ we have the isomorphism  of $\overline{\n}[\whn]$-modules
\beq\label{zakarakter}
W_{N^{\whn}(\Lambda)}^T \cong U(\overline{n}[\whn]_{<0}) ,\qquad\text{where}\qquad
\overline{n}[\whn]_{<0} =\coprod_{m \in \frac{1}{2}\mathbb{Z}_{<0},\,\alpha\in R_+}\mathbb{C}{x_{\alpha}}_{(2m)} \otimes t^m.
\eeq

\section{Quasi-particle bases of principal subspaces}
From now on, we consider  the rectangular  highest weights $\Lambda$, as defined by \eqref{weight}.
We use the realization of the standard $\wtg[\whn]$-module $L^{\whn}(\Lambda) \cong U(\tilde{\mathfrak{g}}[\hat{\nu}]) \cdot v_{ \Lambda}$   as a submodule of $(V_Q^{T})^{\otimes k}$ with the highest weight vector 
\beq \label{jt}
v_{ \Lambda}=v_{ \Lambda_{j_k}}\otimes \cdots \otimes v_{ \Lambda_{j_1}}, \qquad \text{where} \qquad 
j_s=  \begin{cases}  0  & \text{for }  1   \leq   s  \leq  k_0,\\ 
j  & \text{for }   k_0+1  \leq  s \leq  k. \end{cases}
\eeq

\subsection{Twisted quasi-particles}\label{sdf21}
For  positive integers $r$ and $i=1,\ldots ,l$ and $m \in \frac{1}{2}\mathbb{Z}$ we define a {\em quasi-particle} of   {\em charge $r$}, {\em color $i$}  and {\em energy $-m$} as  the coefficient $x_{r\alpha_i}^{\whn}(m)$ of the twisted vertex operator $x_{r\alpha_i}^{\whn}(z)$, as  given by \eqref{kv1} (cf. \cite{BuS,G}).  Our goal is to construct the so-called quasi-particle bases of the principal subspaces, i.e. the bases consisting of some monomials of quasi-particles applied on the highest weight vector.
Charges of quasi-particles which appear in the basis of the principal subspaces $W^T_{L^{\whn}(\Lambda)}$ will be less than or equal to   $k$    since we have the integrability relations
\beq\label{integrabilnost}
x_{(k+1)\alpha_i}^{\whn}(z)=0
\eeq
on the standard $\wtg[\whn]$-modules $L^{\whn}(\Lambda)$ of level  $k$ (cf. \cite{Li}). On the other hand, the elements of basis for $W^T_{N^{\whn}(\Lambda)}$ will contain   quasi-particle of all charges. 

By  \eqref{kv1}, it  follows that  in the case of $W^T_{N^{\whn}(\Lambda)}$ all quasi-particles satisfy
\beq\label{kv2}
x_{r\alpha_i}^{\whn}(z)v_{\Lambda} \in  \begin{cases}
z^{-\frac{r}{2}}W^T_{N^{\whn}(\Lambda)}\left[\left[ z^{1/2}\right]\right], & \text{when }  \nu \alpha_i \neq \alpha_i,\\
  W^T_{N^{\whn}(\Lambda)}\left[\left[ z\right]\right], & \text{when }   \nu \alpha_i = \alpha_i, 
\end{cases} 
\eeq
while in the case of $W^T_{L^{\whn}(\Lambda)}$   we have
\beq\label{kv3}
x_{r\alpha_i}^{\whn}(z)v_{\Lambda} \in  \begin{cases}
z^{-\frac{r}{2}+\frac{1}{2}\sum_{s=1}^r\delta_{i,j_s}}W^T_{L^{\whn}(\Lambda)}\left[\left[ z^{1/2}\right]\right], & \text{when }  \nu \alpha_i \neq \alpha_i,\\
  W^T_{L^{\whn}(\Lambda)}\left[\left[ z\right]\right], & \text{when }  \nu \alpha_i = \alpha_i. 
\end{cases}  
\eeq

The quasi-particle bases of   principal subspaces  will be expressed in   the form
\begin{align}
bv_{\Lambda}&= b^{\whn}(\alpha_{l})\cdots b^{\whn}(\alpha_{1})v_{\Lambda} \nonumber\\
&=\underbrace{x^{\whn}_{n_{r_{l}^{(1)},l}\alpha_{l}}(m_{r_{l}^{(1)},l}) \cdots  x^{\whn}_{n_{1,l}\alpha_{l}}(m_{1,l})}_{b^{\whn}(\alpha_{l})}\cdots 
\underbrace{x^{\whn}_{n_{r_{1}^{(1)},1}\alpha_{1}}(m_{r_{1}^{(1)},1}) \cdots  x^{\whn}_{n_{1,1}\alpha_{1}}(m_{1,1})}_{b^{\whn}(\alpha_{1})}v_{\Lambda},\label{kv4}
\end{align}
where for $i=1,\ldots ,l$ and $t \in \mathbb{N}$ we have
$$
0 \leq n_{r_{i}^{(1)},i}\leq \cdots \leq  n_{1,i}, 
\qquad 
r^{(1)}_{i}\geq r^{(2)}_{i}\geq \cdots \geq  r^{(t)}_{i}\geq 0, 
\qquad 
m_{r_{i}^{(1)},i}\leq   \cdots \leq  m_{1,i}  .
$$
Here    $r^{(p)}_{i}$  for $1 \leq p \leq t$  represent the parts of the partition $\Rc_i=(r_{i}^{(1)}, r_{i}^{(2)}, \ldots , r_{i}^{(t)})$ of some fixed positive number $r_i$. The partition $\Rc_i$ is conjugate to the partition $\Rc'_i=(n_{r_{i}^{(1)},i},\ldots , n_{1,i})$. We now recall   terminology from \cite{BuS, G}. For every color $i=1,\ldots ,l$,   the number $r_i$ is said to be a {\em color-type}, a partition  $\Rc'_i$ a {\em charge-type}, a partition  $\Rc_i$ a {\em dual-charge-type} and a finite sequence of energies $\Ec_i=(m_{r_{i}^{(1)},i}, \ldots, m_{1,i})$ an {\em energy-type} of the monomial $b_i=b^{\whn}(\alpha_{i})$. Generalizing this notions to the quasi-particle monomial $b=b^{\whn}(\alpha_{l})\cdots b^{\whn}(\alpha_{1})$  we define its {\em color-type}  
$\mathcal{C}=\left(r_{l},\ldots, r_{1}\right)$, 
{\em charge-type}   $\Rc'=\left(\Rc'_l, \ldots ,\Rc'_1\right)$,
{\em dual-charge-type}    $\Rc=\left(\Rc_l, \ldots ,\Rc_1\right)$ 
and
{\em energy-type}    
$\Ec=\left(\Ec_l, \ldots, \Ec_1\right)$. Moreover, its {\em total charge} $\chg   b$ and {\em total energy} $\en b$ are defined by
$$ 
\chg   b=\sum_{i=1 }^l r_i\qquad\text{and}\qquad
 \en b=-m_{r_{l}^{(1)},l}- \cdots - m_{1,l}-\cdots - m_{r_{1}^{(1)},1}-  \cdots -m_{1,1}.
$$

Let $M_{QP}$ denote  the set of all quasi-particle monomials as in (\ref{kv4}). We introduce the linear order "$<$'' on the set of all monomials from $M_{QP}$ of fixed degree and $\h_{(0)}$-weight  as follows. Let $b$ and $\overline{b}$ be quasi-particle monomials   of  color-types $\mathcal{C}$ and   $\overline{\mathcal{C}}$, charge-types $\mathcal{R}'$ and   $\overline{\mathcal{R}'}$ and   energy-types $\mathcal{E}$ and  $\overline{\mathcal{E}}$, respectively. Then we write $bv_{\Lambda} < \overline{b}v_{\Lambda}$ if one of the following conditions holds:  
\begin{enumerate}
\item\label{xxx1}  $\chg b > \chg  \overline{b}$
\item\label{xxx2} $\chg  b = \chg  \overline{b}$\quad\text{and}\quad$\mathcal{C}<\overline{\mathcal{C}}$
\item\label{xxx3}   $\mathcal{C}=\overline{\mathcal{C}}$\quad\text{and}\quad$\en b < \en \overline{b}$
\item\label{xxx4}   $\mathcal{C}=\overline{\mathcal{C}}, \quad  \en b = \en \overline{b}$\quad\text{and}\quad$\mathcal{R}'<\overline{\mathcal{R}'}$
\item\label{xxx5} $\mathcal{R}'=\overline{\mathcal{R}'}$\quad\text{and}\quad$\mathcal{E}<\overline{\mathcal{E}}$,
\end{enumerate}
where for integer  sequences   we write
 $(x_p,\ldots ,x_1)< (y_r,\ldots ,y_1)$ 
if there exists $s$ such that 
\beq\label{order2}x_1=y_1,\,\ldots\,,\,x_{s-1}=y_{s-1}\qquad \text{and} \qquad s=p+1< r\quad \text{or}\quad x_s<y_s.
\eeq

\subsection{Quasi-particle bases of principal subspaces}
By  \eqref{kv2}  and  \eqref{kv3},    the vectors of the form  \eqref{kv4},   such that for $i=1,\ldots , l$ their quasi-particle energies satisfy
\beqn\label{rel1}
m_{p,i} \leq -\mu_i n_{p,i}  -\sum_{s=1}^{n_{p,i}}\mu_i \delta_{i, j_s} \qquad\text{with} \qquad  1\leq p \leq r_i^{(1)}\quad\text{and}\quad
 \mu_i\coloneqq\textstyle\frac{1}{2}\langle{\alpha_i}_{(0)}, {\alpha_{i}}_{(0)}\rangle ,
\eeqn
span entire principal subspace.
We shall now strengthen the constraints above by using relations among quasi-particles. 
First, by employing the commutator formula for twisted vertex operators on $N^{\whn}(\Lambda)$ we generalize  \cite[Lemma 4.2]{BuS} as follows. 
\begin{lem}\label{rel3}
Let   $  n_1, n_2  $ be positive integers and $\alpha,\beta\in R_+$ positive simple roots.
\begin{itemize}\label{rel4}
\item[a)] If $\langle \alpha, \beta \rangle=-1=\langle \nu \alpha, \beta \rangle$, we have 
\beqn\label{rel5}
(z_1-z_2)^{\text{min}\{n_1,n_2\}}x^{\whn}_{n_1\alpha}(z_1)x^{\whn}_{n_2\beta}(z_2)=(z_1-z_2)^{\text{min}\{n_1,n_2\}}x^{\whn}_{n_2\beta}(z_2)x^{\whn}_{n_1\alpha}(z_1).
\eeqn
\item[b)] If $\langle \alpha, \beta \rangle=-1$ and $\langle \nu \alpha, \beta \rangle=0$, we have \beqn\label{rel6}(z_1^{1/2}-z_2^{1/2})^{\text{min}\{n_1,n_2\}}x^{\whn}_{n_1\alpha}(z_1)x^{\whn}_{n_2\beta}(z_2)=(z_1^{1/2}-z_2^{1/2})^{\text{min}\{n_1,n_2\}}x^{\whn}_{n_2\beta}(z_2)x^{\whn}_{n_1\alpha}(z_1). \eeqn
\end{itemize}
\end{lem}
The interaction between quasi-particles of the same color  \cite[Lemma 4.1]{BuS} can be summarized as follows. 
\begin{lem}\label{rel7}
Let $i=1,\ldots ,l$ and $v\in W_{V}^T$ for 
$V=N^{\whn}(\Lambda), L^{\whn}(\Lambda)$.
For any charges $n_1, n_2$ such that $  n_1\geq n_2$ and  integer $M=m_1+m_2 $ the $2\mu_i n_2$ vectors
 \begin{align*}
&x^{\whn}_{n_1\alpha_i}(m_1)x^{\whn}_{n_2\alpha_i}(m_2)v ,\,\,x^{\whn}_{n_1\alpha_i}(m_1-1)x^{\whn}_{n_2\alpha_i}(m_2+1)v ,
\ldots\\
&\qquad\qquad\qquad\ldots,   x^{\whn}_{n_1\alpha_i}(m_1-2\mu_i n_2+1)x^{\whn}_{n_2\alpha_i}(m_2+2\mu_i n_2-1)v 
\end{align*}
can be expressed as a  finite  linear combination of the monomial  vectors
$$ 
x^{\whn}_{n_1\alpha_i}(s_1)x^{\whn}_{n_2\alpha_i}(s_2)v
\qquad\text{such that}
\qquad
 s_1+s_2=M\quad\text{and}\quad s_1 \leq m_1-2\mu_i n_2  \text{ or } s_2 < m_2 $$
and monomial vectors which contain a quasi-particle of color $i$ and charge $n_1+1$. 
Moreover, if $n_1=n_2$,  the monomial vectors 
$$x^{\whn}_{n_1\alpha_i}(m_1)x^{\whn}_{n_1\alpha_i}(m_2)v   \qquad\text{such that} \qquad   m_1+m_2=M \quad\text{and}\quad m_2-2\mu_in_1< m_1\leq m_2$$
can be expressed as a finite linear combination of monomial vectors
$$x^{\whn}_{n_1\alpha_i}(s_1)x^{\whn}_{n_1\alpha_i}(s_2)v   \qquad\text{such that} \qquad  s_1+s_2=M \quad\text{and}\quad  s_1\leq s_2-2\mu_in_1$$
and monomial vectors  which contain a quasi-particle of color $i$ and charge $n_1+1$. 
\end{lem}
Denote by $B_{W_{N^{\whn}(\Lambda)}}$ the set of all quasi-particle monomials of the form  \eqref{kv4} such that their energies $m_{p,i}$ satisfy the following difference conditions for all $i=1,\ldots ,l$:
\begin{align} 
m_{p,i} \leq & \,(1-2p)\mu_i n_{p,i} -\langle{\alpha_i}_{(0)}, {\alpha_{i-1}}_{(0)}\rangle\sum_{q=1}^{r_{i-1}^{(1)}}\min \textstyle\left\{n_{q,i-1},n_{p,i}\right\}-\displaystyle\sum_{s=1}^{n_{p,i}}\mu_i \delta_{i, j_s},
     \label{rel8}\\
m_{p+1,i}& \leq m_{p,i} -2\mu_i n_{p,i}  \quad \text{if} \quad n_{p+1,i}=n_{p,i},  \label{rel9}
\end{align}
where  we set $r_0^{(1)}\coloneqq 0$ and $j_s$ are given by \eqref{jt}. As for the standard module, we denote by $B_{W_{L^{\whn}(\Lambda)}}$ the subset of $B_{W_{N^{\whn}(\Lambda)}}$ such that its monomials contain only quasi-particles of charge less than or equal to $k$; recall \eqref{integrabilnost}. For $V=N^{\whn}(\Lambda),L^{\whn}(\Lambda)$ we have
\begin{thm}\label{t1}
For any highest weight $\Lambda$ as in \eqref{weight} the set $\mathcal{B}_{W_V}=\{bv_{\Lambda} \,| \,b \in B_{W_V} \}$ forms a basis of the principal space $W_{V }^T$.
\end{thm}

The proof that $\mathcal{B}_{W_V}$ spans the principal subspace goes   in parallel with \cite[Sect. 5.5]{BK} and \cite[Sect. 5]{G}. Furthermore, the linear independence of $\mathcal{B}_{W_{N^{\whn}(\Lambda)}}$ can be verified by arguing as in \cite[Sect. 3]{Bu1}. As for the set   $\mathcal{B}_{W_{L^{\whn}(\Lambda)}}$, in comparison with \cite{BuS} (where the analogous basis for the highest weights of the form $k\Lambda_0$ is established), 
the proof of   linear independence in the rectangular case requires a certain new ingredient, the map $\left[\pi_0\lambda_j\right]$ which connects modules $L^{\whn}(\Lambda_0)$ and $L^{\whn}(\Lambda_j)$. In Subsection  \ref{ghj1}  below, we recall the Georgiev  projection $\pi_{\Rc}$ and  the aforementioned map $\left[\pi_0\lambda_j\right]$. Finally, in Subsection  \ref{ghj3} we employ them to finalize the proof.

\subsection{Georgiev's projection and the simple current map }\label{ghj1}
Let $\Rc$ be the dual-charge-type  of the quasi-particle monomial from  \eqref{kv4}.  Consider    Georgiev's projection $\pi_{\Rc}$ of the principal subspace  $W_{L^{\whn}(\Lambda) }^T$ on the tensor product space
$$
{W_{L^{\whn}(\Lambda_{j_k}) }^T}_{(r_l^{(k)}, \ldots, r_1^{(k)})} \otimes  \cdots  \otimes {W_{L^{\whn}(\Lambda_{j_1}) }^T}_{(r_l^{(1)}, \ldots, r_1^{(1)})} \subset  
  ({W_{L^{\whn}(\Lambda_{j}) }^T})^{\otimes k_j}\otimes ({W_{L^{\whn}(\Lambda_{0}) }^T})^{\otimes k_0} ,
$$
where   ${W^T_{L^{\whn}(\Lambda_{j_s})}}_{(r_l^{(s)}, \ldots, r_1^{(s)})}$, $s=1,\ldots ,   k$,
denote  the $\mathfrak{h}_{(0)}$-weight subspaces of the level one principal subspace $W^T_{ L^{\whn}(\Lambda_{j_s})}$ consisting of vectors of charges $r_l^{(s)}, \ldots, r_1^{(s)}$; see \cite{BuS,G} for more details. We will also use the symbol $\pi_{\Rc}$ to denote  the generalization of the projection   to the space of formal power series with coefficients in 
$({W_{L^{\whn}(\Lambda_{j}) }^T})^{\otimes k_j}\otimes ({W_{L^{\whn}(\Lambda_{0}) }^T})^{\otimes k_0}$. 
As in \cite[Sect. 5.1]{BuS},
one can use the  integrability relations \eqref{integrabilnost} at the level $k=1$, $x_{2\alpha_i}^{\widehat{\nu}}(z)^2=0$ for $i=1,\ldots ,l$,    to express the     projection of the monomial \eqref{kv4} as the coefficient of the   product of the corresponding vertex operators applied on the highest weight vector \eqref{jt}.

Denote by $\left[\pi_0\lambda_j\right]$ the identity map on the vector space $V_Q^T$ endowed with two different structures of level one standard $\wtg[\whn]$-modules, $L^{\whn}(\Lambda_0)$ and $L^{\whn}(\Lambda_j)$ (cf. \cite{DLM,Li,Li1}). By the construction of these modules follows that, for the suitably normalized highest weight vector $v_{\Lambda_j}$, we have for any root $\alpha$, $m \in \frac{1}{2}\mathbb{Z}$ and $i=1,\ldots ,l$  the   identities 
\begin{gather}
\left[\pi_0\lambda_j\right] v_{\Lambda_0}=v_{\Lambda_j},\label{id1}\\
x^{\whn}_{\alpha}(m)\left[\pi_0\lambda_j\right] = \left[\pi_0\lambda_j\right]x^{\whn}_{\alpha} ( m+ \langle{\lambda_j}_{(0)}, \alpha_{(0)}\rangle ),\label{id2}\\
\left[\pi_0\lambda_j\right] \circ e_{\alpha_i}= e_{\alpha_i} \circ \left[\pi_0\lambda_j\right] .\label{id21}
\end{gather}

Suppose $b$ is the quasi-particle monomial from \eqref{kv4}.
The equalities \eqref{id1}  and  \eqref{id2}  imply that  $\pi_{\mathcal{R}}b(\left[\pi_0\lambda_j\right] v_{\Lambda_0})^{\otimes k_j}\otimes v_{\Lambda_0}^{\otimes k_0} $ coincides with the coefficient $(\left[\pi_0\lambda_j\right]^{\otimes  k_j}\otimes 1^{\otimes k_0})\left(\pi_{\mathcal{R}}b^+v_{\Lambda_0} ^{\otimes k}\right)$  of the variables  
\begin{align*}
&  z_{r_i^{(1)},i}^{-m_{r_i^{(1)},i}-n_{r_i^{(1)},i }} \cdots  z_{r^{(k_0+1)}_{i}+1,i}^{-m_{r^{(k_0+1)}_{i}+1,i}-n_{r^{(k_0+1)}_{i}+1,i}}\cdots\\
&\cdots z_{r^{(k_0+1)}_{i},i}^{-m_{r^{(k_0+1)}_{i},i}-n_{r^{(k_0+1)}_{i},i}+\left<{\lambda_j}_{(0)}, {\alpha_i}_{(0)}\right>(n_{r^{(k_0+1)}_{i},i}-k_0)}\cdots z_{1,i}^{-m_{1,i}-n_{1,i}+\left<{\lambda_j}_{(0)}, {\alpha_i}_{(0)}\right>(n_{1,i}-k_0)}
\end{align*}
in the image of the (generalized) Georgiev  projection
$$\pi_{\mathcal{R}} \Big( x^{\whn}_{n_{r_{l}^{(1)},l}\alpha_{l}}(z_{r_{l}^{(1)},l})\cdots   x^{\whn}_{n_{1,1}\alpha_{1}}(z_{1,1}) \ (\left[\pi_0\lambda_j\right]v_{\Lambda_0})^{\otimes k_j}\otimes v_{ \Lambda_0}^{\otimes k_0}\Big).
$$
Here we denote by $b^+$ the quasi-particle monomial
\beq\label{id5}
b^+ = b^{\whn +}_{\alpha_l}\,\cdots \,b^{\whn +}_{\alpha_1}\eeq
such that its factors $b^{\whn +}_{\alpha_i}$, $i=1,\ldots ,l$,  are given by
\begin{align} 
b^{\whn +}_{\alpha_i}
=&\,x_{n_{r^{(1)}_{i},i}\alpha_{i}}(m_{r^{(1)}_{i},i})\cdots  x_{n_{r^{(k_0+1)}_{i}+1,i}\alpha_{i}}(m_{r^{(k_0+1)}_{i}+1,i})\cdots\nonumber\\
& \cdots x_{n_{r^{(k_0+1)}_{i},i}\alpha_{i}} (m_{r^{(k_0+1)}_{i},i}+\langle{\lambda_j}_{(0)}, {\alpha_i}_{(0)}\rangle(n_{r^{(k_0+1)}_{i},i}-k_0) )\cdots\nonumber\\
 &\cdots  x_{n_{1,i}\alpha_{i}} (m_{1,i}+\langle{\lambda_j}_{(0)}, {\alpha_i}_{(0)}\rangle(n_{1,i}-k_0) ).\label{id6}
\end{align}
Finally, it remains to observe that by \eqref{rel8} and \eqref{rel9} the energies of   $b^+v_{k\Lambda_0}$ satisfy the initial and difference conditions  for the  quasi-particle basis of  principal subspace $W_{L^{\whn}(k\Lambda_0) }^T$, as given by \cite[Thm. 5.1]{BuS}.

\subsection{Proof of linear independence}\label{ghj3}
We are now ready to complete the proof of Theorem \ref{t1}.
Suppose that the quasi-particle monomial $b$ of charge-type $\mathcal{R}'$ and dual-charge-type $\mathcal{R}$ is the smallest monomial, with respect to the linear ordering from Subsection \ref{sdf21}, in the linear combination (indexed by some finite set $A$),
\begin{equation}\label{lin1}
\sum_{a\in A}c_ab_av_{\Lambda}=0.
\end{equation}
We can assume that all quasi-particle monomials  $b_a \in B_{W_{L^{\whn}(\Lambda)}}$ are of the same color-type and that the scalars $c_a \in \mathbb{C}$ are nonzero.    Let us apply the projection $\pi_{\mathcal{R}}$ to  \eqref{lin1}. By its definition,  all nonzero summands which appear in 
\beqn\label{p7}
\sum_{a\in A}c_a\pi_{\mathcal{R}}b_av_{\Lambda}=0 
\eeqn
contain
monomials
 $b_a$ of the same charge-type $\mathcal{R}'$.
Using \eqref{id1} we find
\beqn
0=\sum_{a \in A}
c_{a}\pi_{\mathcal{R}}b_a(\left[\pi_0\lambda_j\right] v_{\Lambda_0})^{\otimes k_j}\otimes v_{ \Lambda_0}^{\otimes k_0}=(\left[\pi_0\lambda_j\right] ^{\otimes k_j}\otimes 1^{\otimes k_0})\sum_{a \in A}
c_{a}\pi_{\mathcal{R}}b^{+}_a v_{ \Lambda_0}^{\otimes k} 
\eeqn
with   $b^+_a$   as in \eqref{id5} and \eqref{id6}.
Dropping the invertible operator $\left[\pi_0\lambda_j\right] ^{\otimes  k_j}\otimes 1^{\otimes k_0}$, we   get 
\beqn \label{eq:d40}
\sum_{a \in A}
c_{a}\pi_{\mathcal{R}}b^{+}_a v_{ \Lambda_0}^{\otimes k}=0.
\eeqn
Thus, using the quasi-particle basis for the weights of the form $k\Lambda_0$, as given by \cite[Thm. 5.1]{BuS}, we obtain $c_a=0$, so that the 
 assertion of   Theorem \ref{t1} follows.

\section{Combinatorial bases of standard modules and parafermionic spaces}

In this section, we use the quasi-particle bases for prinicipal subspaces to construct combinatorial bases of standard modules and their parafermionic spaces. Throughout the section, $\Lambda$ denotes the rectangular highest weight; recall \eqref{weight}.

\subsection{Standard modules}

Let  $B_{U(\whh[\nu]^{-})}$ be the Poincar\'{e}--Birkhoff--Witt  basis   of the irreducible $\whh[\nu]_{\frac{1}{2}\mathbb{Z}}$-module $U(\whh[\nu]^{-})$ of level $k$ which consists of all elements
$h=h_{\alpha_l}\cdots h_{\alpha_1}$ such that for $i=1,\ldots ,l$
we have
\begin{gather}
h_{\alpha_i}=({\alpha_i}_{(2s_{t_i,i})}(-s_{t_i,i}))^{r_{t_i,i}}\ldots ({\alpha_i}_{(2s_{1,i})}(-s_{1,i}))^{r_{1,i}},\qquad\text{where}\qquad \label{oznaka}\\
t_i\in\mathbb{Z}_{\geq 0}\quad\text{and}\quad
r_{p,i}, s_{p,i}\in \mathbb{N}\quad\text{are such that}\quad s_{1,i}\leq  \ldots \leq  s_{t_{i},i} \quad \text{for all}\quad
 p =1,\ldots , t_i.\nonumber
\end{gather}
We shall now   generalize the definition of the linear order from Subsection \ref{sdf21} to the set 
\beq\label{starr}
\left\{ e_{\mu}\ts h\ts b \ts v_{\Lambda}  \,\,\big|\big.\,\, \mu \in Q, h \in  B_{U(\whh[\nu]^{-})}, b \in M_{QP}'\right\},
\eeq
where $M'_{QP}$ denotes the subset of $M_{QP}$ which consists of all   vectors of the form  \eqref{kv4} such that their quasi-particle charges   are less than or equal to $k-1$. Let  $e_{\mu}hbv_{\lambda}$ and $e_{\mu'} \overline{h} \overline{b}v_{\lambda}$ be any two elements of the set \eqref{starr} of the same degree and $\h_{(0)}$-weight. We shall write $e_{\mu}hbv_{\lambda} < e_{\mu'} \overline{h} \overline{b}v_{\lambda}$ if  we have $bv_{\lambda} <  \overline{b}v_{\lambda}$ (i.e., in other words, if one of the conditions \eqref{xxx1}--\eqref{xxx5} holds) or if the energy types    of $b$ and $\overline{b}$ coincide and we have $h< \overline{h}$. The order on   $B_{U(\whh[\nu]^{-})}$ is defined as follows. For any two elements 
$h=h_{\alpha_l}\cdots h_{\alpha_1}$ and     
$\overline{h} =\overline{h}_{ \alpha_l}\cdots \overline{h}_{ \alpha_1}$ of the basis $ B_{U(\whh[\nu]^{-})}$
 we write $h< \overline{h}$ if one of the  conditions holds:
\begin{enumerate}[(a)]
\item\label{xxxx1}   $(r_{t_l,l}, \ldots, r_{1,1}) < (\overline{r}_{\overline{t}_l,l}, \ldots, \overline{r}_{1,1}) $  
\item\label{xxxx2}    $(r_{t_l,l}, \ldots, r_{1,1}) = (\overline{r}_{\overline{t}_l,l}, \ldots, \overline{r}_{1,1}) $\,\, and\,\,  $(s_{t_l}, \ldots, s_1) < (\overline{s}_{\overline{t}_l}, \ldots, \overline{s}_1) $.
\end{enumerate} 
The   integer sequences in \eqref{xxxx1} and \eqref{xxxx2} have the same meaning as in \eqref{oznaka} and we compare them as in \eqref{order2}.

In order to construct the quasi-particle basis of  standard module $L^{\whn}(\Lambda)$ we need the canonical isomorphism of $d$-graded vector spaces given by Lepowsky-Wilson in \cite{LW2},
\begin{align}
U(\whh[\nu] ^{-}) \otimes L^{\whn}(\Lambda)^{\whh[\nu] ^{+}} \,&\xrightarrow{\hspace{5pt}\cong\hspace{5pt}}\, L^{\whn}(\Lambda)\label{LW}\\
\nonumber h\otimes u\,\,&\xmapsto{\hspace{17pt}}\,\, h\cdot u,
\end{align}
where $L^{\whn}(\Lambda)^{\whh[\nu] ^{+}}$ denotes the {\em vacuum subspace} of standard module  $L^{\whn}(\Lambda)$, i.e. 
\beq\label{vacuum_space}
L^{\whn}(\Lambda)^{\whh[\nu] ^{+}} =\left\{v\in L^{\whn}(\Lambda)\,\big|\big.\,\, \whh[\nu] ^{+} \hspace{-1pt}\cdot\hspace{-1pt} v=0\right\} .
\eeq
Moreover, the construction relies the vertex operator formula 
\beq\label{star}
\frac{1}{p!}(2zx^{\whn}_{\alpha}(z))^{p}=\frac{1}{q!}\epsilon_{C_0}(\alpha, -\alpha)^{-q}
E^{-}(-\alpha, z)(2zx^{\whn}_{-\alpha}(z))^qE^{+}(-\alpha, z)e_{\alpha}z^{\mu_{\alpha}+\alpha_{(0)}},
\eeq
 where $p,q \geq  0$ are such that  $k=p+q$  and $\mu_{\alpha}= \left<\alpha_{(0)}, \alpha_{(0)}\right>/2$  (cf. \cite{OT}). Formula \eqref{star} implies that the quasi-particle of color $i$ and charge $k$ acts on $L^{\whn}(\Lambda)$ as an operator
\beqn\label{star1}
 x^{\whn}_{k\alpha_i}(m)=e_{\alpha_i}^kh\quad\text{for some}\quad h \in U(\whh[\nu]).
\eeqn
Consequently, we do not need quasi-particles of the highest charge $k$ to form the   basis of the standard module.
Let $Q_{(0)}=\left\{\alpha_{(0)}\,|\,\alpha\in Q\right\}$.
We have
\begin{thm}\label{t2}
For any highest weight $\Lambda$ as in \eqref{weight}  the set $$\mathcal{B}_{L^{\whn}(\Lambda)}= \left\{ e_{\mu}\ts h\ts b\ts v_{ \Lambda} \,\,\big|\big.\,\, \mu \in Q_{(0)},\, h \in B_{U(\whh[\nu]^{-})},\, b \in B_{W_{L^{\whn}(\Lambda)}}\cap M'_{QP}\right\}$$
 forms a basis of the standard module $L^{\whn}(\Lambda)$.
\end{thm}

\begin{proof}
The proof that   $\mathcal{B}_{L^{\whn}(\Lambda)}$ spans the standard module $L^{\whn}(\Lambda)$ relies on the quasi-particle  relations       \eqref{star}  and Lemmas \ref{rel3} and \ref{rel7}. It can be carried out by arguing as in \cite[Lemma 2.3]{BKP}, so we omit its details.

As for  the linear independence, its proof employs  the Georgiev-type projection 
\beqn\label{proj}
\pi_{\mathcal{R}_{\alpha_1}}: \big(L(\Lambda)\big)_{r_1 } \to 
L(\Lambda_{j_k})_{r_1^{(1)}} \otimes \ldots  \otimes L(\Lambda_{j_1})_{r_1^{(k)}},
\eeqn
with respect to the decomposition 
\begin{align}
&L^{\whn}(\Lambda)=\coprod_{r\in\mathbb{Z}} L^{\whn}(\Lambda)_r \subset  L^{\whn}(\Lambda_{j_k })  \otimes\ldots \otimes     L^{\whn}(\Lambda_{j_1}) , \qquad\text{where}\label{proj1}\\
&L^{\whn}(\Lambda)_r=\coprod_{r_2, \ldots, r_l\in\mathbb{Z}} L^{\whn}(\Lambda)_{k\Lambda\arrowvert_{\h_{(0)}}+r_l\alpha_l+\cdots +r_2\alpha_2+r\alpha_1}.\nonumber
\end{align}
The projection $\pi_{\mathcal{R}_{\alpha_1}}$, as well as its generalization to the space of formal power series with coefficients in \eqref{proj1} (which we again denote by the same symbol), are uniquely determined by the dual-charge type  $\textstyle\mathcal{R}_{\alpha_1}= (r_1^{(1)}, r_1^{(2)}, \ldots, r_1^{(k)} )$ and the  color-type $r_1=\sum_{t=1}^{k_{\alpha_1}} r_1^{(t)}$ with respect to the color $i= 1$. Consider   the projection of the vector
\begin{align*}
&e_{\mu}\,
({\alpha_l}_{(2s_{t_l,l})}(-s_{t_l,l}))^{r_{t_l,l}}\ldots ({\alpha_1}_{(2s_{1,1})}(-s_{1,1}))^{r_{1,1}}\\
&\times x^{\whn}_{n_{r_{l}^{(1)},l}\alpha_{l}}(m_{r_{l}^{(1)},l}) \ldots  x^{\whn}_{n_{1,l}\alpha_{l}}(m_{1,l})  
\ldots
x^{\whn}_{n_{r_{1}^{(1)},1}\alpha_{1}}(m_{r_{1}^{(1)},1}) \ldots  x^{\whn}_{n_{1,1}\alpha_{1}}(m_{1,1}) \,v_{ \Lambda} ,
\end{align*}
such that its submonomial with respect to color $i=1$,
$$b^{\whn}(\alpha_1)=x^{\whn}_{n_{r_{1}^{(1)},1}\alpha_{1}}(m_{r_{1}^{(1)},1}) \ldots  x^{\whn}_{n_{1,1}\alpha_{1}}(m_{1,1}),$$
is of dual-charge-type $\mathcal{R}_{\alpha_1}$. It  coincides with the corresponding  coefficient in
\begin{align*}
\pi_{\mathcal{R}_{\alpha_1}} 
&e_{\mu}\,
(\alpha_{l,-}(w_{t_l,l}))^{r_{t_l,l}}\ldots ( \alpha_{1,-}(w_{1,1}))^{r_{1,1}}\,
 \\
&\times x^{\whn}_{n_{r_{l}^{(1)},l}\alpha_{l}}(z_{r_{l}^{(1)},l}) \ldots  x^{\whn}_{n_{1,l}\alpha_{l}}(z_{1,l})
\ldots
x^{\whn}_{n_{r_{1}^{(1)},1}\alpha_{1}}(z_{r_{1}^{(1)},1}) \ldots  x^{\whn}_{n_{1,1}\alpha_{1}}(z_{1,1}) \, v_{ \Lambda },
\end{align*} 
where $\alpha_{i,-} (z)=\sum_{m<0} {\alpha_i}_{(2m)} (m) z^{-m-1}$ for $i=1,\ldots ,l$.

Suppose that there exists a finite linear combination 
\beq \label{maindk1}
\sum c_{\mu, h, b}\ts e_{\mu}\ts h\ts b\ts v_{ \Lambda}=0
\eeq
of vectors    $e_{\mu}  h  b  v_{ \Lambda}\in\mathcal{B}_{L^{\whn}(\Lambda)}$ 
of   the same degree   and $\h_{(0)}$-weight   such that all coefficients $c_{\mu, h, b}\in\mathbb{C}$ are nonzero. Choose any monomial vector in \eqref{maindk1}  of the form
\beq \label{maindk2}
e_{\mu }\ts h \ts b \ts v_{\Lambda} \qquad\text{with}\qquad {\chg b (\alpha_1)} =r_1 
\eeq
and   the maximal charge-type $\mathcal{R}'_{\alpha_1}$.  As before,   in \eqref{maindk2}  we write $b (\alpha_1)$ for the submonomial of $b $ in color $i=1$.
Its dual-charge-type 
$ 
\mathcal{R}_{\alpha_1}=\textstyle (r_1^{(1)}, \ldots, r_1^{(p)} )
$ 
in color $i=1$, where $p < k$ and $r_1=r_1^{(1)}+ \ldots +r_1^{(p)}$, uniquely determines the Georgiev-type projection $\pi_{\mathcal{R}_{\alpha_1}}$ (with $r_1^{(t)}=0$ for $t > p$). The projection $\pi_{\mathcal{R}_{\alpha_1}}$ ensures that only  nonzero summands in
\beq \label{maindk3}
\sum c_{\mu, h, b}\ts \pi_{\mathcal{R}_{\alpha_1}}\ts e_{\mu}\ts h\ts b\ts v_{ \Lambda}=0 
\eeq
are   those of the form as in  \eqref{maindk2}. Indeed,   from
$$e_{\alpha_1}\ts (v_{\Lambda_{j_k}}\otimes \cdots \otimes  v_{ \Lambda_{j_1}}) 
=e_{\alpha_1} v_{\Lambda_{j_k} }\otimes \cdots \otimes e_{\alpha_1} v_{ \Lambda_{j_1}}  ,$$
and 
\beqn e_{\mu }\ts h \ts b \ts v_{ \Lambda } \in \coprod_{\substack{r_1, \ldots, r_{k-1}\in\mathbb{Z}\\ r_k>0}} L^{\whn}(\Lambda_{j_k})_{r_1} \otimes \cdots \otimes L^{\whn}(\Lambda_{j_1})_{r_k}, \eeqn
follows that $\pi_{\mathcal{R}_{\alpha_1}}$  annihilates all   vectors $e_{\mu }\ts h \ts b \ts v_{\Lambda}$ in \eqref{maindk1} such that  $\chg b (\alpha_1)  <r_1$. 

By using the identities  \eqref{id1}--\eqref{id21} and arguing as in Subsection \ref{ghj3} we  write  \eqref{maindk3}  as
\beq \label{maindk4}
\sum c_{\mu, h, b}\ts \pi_{\mathcal{R}_{\alpha_1}}  e_{\mu}\ts h\ts b^+\ts v_{ k\Lambda_0}=0,
\eeq
where the monomials $b^+$ are found  as in  \eqref{id5} and \eqref{id6}. Hence we can now proceed with the iterated use of the constant terms of twisted $\Delta$-maps from \cite{Li} and the bijective map $e_{\alpha_1}$ as in \cite[Sect. 5.4]{BuS}, until we reduce  \eqref{maindk4}  to a linear combination of the vectors of the form $\pi_{\mathcal{R}_{\alpha_1}}  e_{\mu}\ts h\ts b \ts v_{ k\Lambda_0}$      which satisfy the   initial and difference conditions and $\chg b (\alpha_1)=0$; see  \cite[Thm. 5.1]{BuS} for more details. Next, we apply the same procedure for the remaining colors $i=2, \ldots , l$. Finally, we find that  all  coefficients $c_{\mu, h, b}$ in  \eqref{maindk1}  are zero, thus verifying the linear independence of the set  $\mathcal{B}_{L^{\whn}(\Lambda)}$.
\end{proof}

\subsection{Vacuum spaces}
Following   J. Lepowsky and R. Wilson \cite{LW1}, for a quasi-particle monomial  of charge-type 
$\Rc'=(n_{r_l^{(1)},l},\ldots ,n_{1,1})$ we define $\mathcal{Z}$-operators
\begin{align*}
\Zc_{\Rc'}(z_{r_l^{(1)},l},\ldots ,z_{1,1})=&\,\,
E^-(\alpha_l,z_{r_l^{(1)},l})^{n_{r_l^{(1)},l}/{k }} \cdots  E^-(\alpha_1,z_{1,1})^{n_{1,1}/{k }}
x^{\whn}_{\Rc'}(z_{r_l^{(1)},l},\ldots ,z_{1,1})\\
\nonumber
&\times E^+(\alpha_l,z_{r_l^{(1)},l})^{n_{r_l^{(1)},l}/{k }} \cdots  E^+(\alpha_1,z_{1,1})^{n_{1,1}/{k }},
\end{align*}
where
$$
x^{\whn}_{\Rc'}(z_{r_l^{(1)},l},\ldots ,z_{1,1})=x^{\whn}_{n_{r_l^{(1)},l}\alpha_l} (z_{r_l^{(1)},l}) \ldots
x^{\whn}_{n_{1,1}\alpha_1} (z_{1,1}).
$$
We shall write the above formal Laurent series as
$$
\Zc_{\Rc'}(z_{r_l^{(1)},l},\ldots ,z_{1,1})=
\sum_{m_{r_l^{(1)},l},\ldots ,m_{1,1}\in\mathbb{Z}}
\Zc_{\Rc'}(m_{r_l^{(1)},l},\ldots ,m_{1,1}) 
z_{r_l^{(1)},l} ^{-m_{r_l^{(1)},l}-n_{r_l^{(1)},l}}\cdots z_{1,1}^{-m_{11}-n_{11}}.
$$
Its coefficients 
act on the vacuum space, i.e. we have
\beq\label{z_action1}
\Zc_{\Rc'}(m_{r_l^{(1)},l},\ldots ,m_{1,1}) 
\colon
L^{\whn}(\Lambda)^{\whh[\nu] ^{+}}\to L^{\whn}(\Lambda)^{\whh[\nu] ^{+}}.
\eeq
By \eqref{e4} we also  have the action of the Weyl group translations $e_{\alpha}$ on the  vacuum space,  
\beqn\label{weyl_action}
e_{\alpha}\colon L^{\whn}(\Lambda)^{\whh[\nu] ^{+}}\to L^{\whn}(\Lambda)^{\whh[\nu] ^{+}}.
\eeqn  
Consider the projection
\beq\label{projection}
\pi^{\whh[\nu] ^{+}}\colon L^{\whn}(\Lambda) \to L^{\whn}(\Lambda)^{\whh[\nu] ^{+}}
\eeq
defined by the direct sum decomposition
$$
L^{\whn}(\Lambda)=L^{\whn}(\Lambda)^{\whh[\nu] ^{+}}\oplus\, \whh[\nu] ^{-} U(\whh[\nu] ^{-})\hspace{-1pt}\cdot\hspace{-1pt} L^{\whn}(\Lambda)^{\whh[\nu] ^{+}} .
$$
By \eqref{z_action1} it follows that  the action of projection \eqref{projection} satisfies
$$
\pi^{\whh[\nu] ^{+}}\colon
x^{\whn}_{\Rc'}(z_{r_l^{(1)},l},\ldots ,z_{1,1}) v_{\Lambda}
\mapsto \Zc_{\Rc'}(z_{r_l^{(1)},l},\ldots ,z_{1,1})v_{ \Lambda}.
$$
By employing the isomorphism  \eqref{LW} along with the projection \eqref{projection}, one easily obtains   the following consequence of   Theorem \ref{t2}:
\begin{thm}\label{t3}
For any highest weight $\Lambda$ as in \eqref{weight} the set of vectors
$$
e_\mu \ts \Zc_{\Rc'}(m_{r_l^{(1)},l},\ldots ,m_{1,1})  v_{\Lambda},
$$
where $\mu\in Q_{(0)}$ and the charge-type $\Rc'$ and the energy-type
$(m_{r_l^{(1)},l},\ldots ,m_{1,1}) $
satisfy the initial and difference conditions for $\mathcal{B}_{W_{L^{\whn}(\Lambda)}}$, forms a basis of the vacuum space $ L^{\whn}(\Lambda)^{\whh[\nu] ^{+}}$.
\end{thm}

\subsection{Parafermionic spaces}

We now extend the definition of the parafermi\-o\-nic space of highest weight $k\Lambda_0$ from \cite{OT} (see also \cite{G2,BKP}   for the untwisted case) to any rectangular highest weight \eqref{weight}.
The {\em parafermi\-o\-nic space} is defined as the quotient
\beq\label{parafermion}
L^{\whn}(\Lambda)_{kQ}^{\whh[\nu] ^{+}}=L^{\whn}(\Lambda)^{\whh[\nu] ^{+}} \Big/ 
\text{span}\left\{(\rho(k\alpha) -1)\cdot v\,\,\big|\big.\,\, \alpha\in Q,\, v\in  L^{\whn}(\Lambda)^{\whh[\nu] ^{+}}\right\},
\eeq
where the map $\rho$ is given by  \eqref{ro}.
Denote by $\pi_{kQ}^{\whh[\nu] ^{+}}$ the canonical projection 
$$
\pi_{kQ}^{\whh[\nu] ^{+}}:L^{\whn}(\Lambda)^{\whh[\nu] ^{+}}\to L^{\whn}(\Lambda)_{kQ}^{\whh[\nu] ^{+}}.
$$
We have  
$$
L^{\whn}(\Lambda)_{kQ}^{\whh[\nu] ^{+}}
\cong\coprod_{\mu\in\Lambda+Q/kQ}L^{\whn}(\Lambda)_{\mu}^{\whh[\nu] ^{+}}.
$$
As in \cite{OT},  we assign to quasi-particle monomials of charge-type $\Rc'=(n_{r_l^{(1)},l},\ldots ,n_{1,1})$ the   $\Psi$-operators
\beqn
\Psi_{\Rc'} (z_{r_l^{(1)}\, l},\ldots , z_{1,1})=
\Zc_{\Rc'}(z_{r_l^{(1)}\, l},\ldots , z_{1,1})
z_{r_l^{(1)}\, l}^{-n_{r_l^{(1)},l}{ \alpha_l}_{(0)}  /k   }  \cdots  z_{1,1}^{-n_{1,1} { \alpha_1}_{(0)}  / k }\epsilon_{\alpha_l}^{-n_{1,l} \alpha_l  / k }\cdots \epsilon_{\alpha_1}^{-n_{1,1} \alpha_1 / k },
\eeqn
where $\epsilon_{\alpha}:=\epsilon_{C_0}(\alpha, \cdot)$ is the $2$-cocycle on the weight lattice $P$. 
In parallel with  \cite[Lemma 12]{OT}, we have
\begin{align}
\Psi_{n_r\beta_r,\ldots ,n_1\beta_1} (z_r,\ldots ,z_1)= 
\prod_{1\leq  p<s\leq  r} &
\bigg(\left(z^{1/2}_s-z^{1/2}_p\right)^{\left<n_s\beta_s,n_p\beta_p\right>/k}  \left(z^{1/2}_s+z^{1/2}_p\right)^{\left<n_s\nu\beta_s,n_p\beta_p\right>/k}\bigg. \nonumber\\
&\bigg. \times\epsilon_{C_0}(\beta_s, \beta_p)^{ n_sn_p/ k }\bigg)  \, \Psi_{n_r\beta_r} (z_r)\ldots \Psi_{n_1\beta_1} (z_1).\label{relpsi}
\end{align}
Since the $\Psi$-operators commute both with the action of the Heisenberg subalgebra $\whh[\nu] _{\frac{1}{2}\mathbb{Z}}$ and with the action of $\rho(kQ)$, their coefficients $\psi_{\Rc'}(m_{r_l^{(1)},l},\ldots ,m_{1,1})$, given by  
$$
\Psi_{\Rc'}(z_{r_l^{(1)},l},\ldots ,z_{1,1})=
\sum_{
m_{r_l^{(1)},l},\ldots ,m_{1,1}}
\psi_{\Rc'}(m_{r_l^{(1)},l},\ldots ,m_{1,1}) 
z_{r_l^{(1)},l} ^{-m_{r_l^{(1)},l}-n_{r_l^{(1)},l}}\ldots z_{1,1}^{-m_{1,1}-n_{1,1}},
$$
are induced operators on the quotient  \eqref{parafermion}. The summation on the right-hand side goes over all sequences $(m_{r_l^{(1)},l},\ldots ,m_{1,1})$ such that $m_{p,i}\in \frac{n_{p,i}}{k} \langle  {\alpha_i}_{(0)},  \mu \rangle+\mu_{i}\mathbb{Z}$ for all $1 \leq p \leq r_i^{(1)}$ when $\psi_{\Rc'}(m_{r_l^{(1)},l},\ldots ,m_{1,1})$ acts on the $\mu$-weight subspace of $L^{\whn}(\Lambda)^{\whh[\nu] ^{+}}$.  The following relations, which connect the coefficients of $\Psi$-operators and   $\mathcal Z$-operators, follow  from \cite[Lemma 10]{OT} (see also    \cite[Lemma 3.2]{BKP}):
\begin{align*} 
&\mathcal Z_{\Rc'} (m_{r_l^{(1)},l},\ldots ,m_{1,1}) \bigg|\bigg._{L^{\whn} (\Lambda)^{\whh[\nu] ^{+}}_\mu}=
\prod_{i=1}^l\prod_{p=1}^{r_i^{(1)}}
\bigg(\epsilon_{C_0}(\textstyle-\frac{1}{k}n_{p,i}\alpha_i, \mu)^{-1} \bigg. \\
&\bigg.\qquad\qquad\qquad\qquad\times\psi_{\Rc'}\textstyle(m_{r_l^{(1)},l}+\frac{1}{k}n_{r_l^{(1)},l}\langle{\alpha_l}_{(0)},\mu\rangle, \ldots , m_{1,1}+\frac{1}{k}n_{1,1}\langle{\alpha_1}_{(0)},\mu\rangle\hspace{-2pt})\bigg|\bigg._{L^{\whn} (\Lambda)^{\whh[\nu] ^{+}}_\mu}\bigg).
\end{align*}
Finally,  Theorem \ref{t3} implies 

\begin{thm}\label{t4}
For any highest weight $\Lambda$ as in \eqref{weight} the set of vectors
\begin{align*}
&\pi_{kQ}^{\whh[\nu] ^{+}} \Zc_{\Rc'}(m_{r_l^{(1)},l},\ldots ,m_{1,1})  v_{\Lambda}\\
&\qquad =
\psi_{\Rc'}(\textstyle m_{r_l^{(1)},l}+\frac{1}{k}n_{r_l^{(1)},l}\langle{\alpha_l}_{(0)},\Lambda\rangle,\ldots ,m_{1,1}+\frac{1}{k}n_{1,1}\langle{\alpha_1}_{(0)},\Lambda\rangle)  v_{\Lambda},
\end{align*}
where    the charge-type $\Rc'$ and the energy-type
$(m_{r_l^{(1)},l},\ldots ,m_{1,1}) $
satisfy the initial and difference conditions for $\mathcal{B}_{W_{L^{\whn}(\Lambda)}}$,
forms a basis of the parafermionic space $L^{\whn} (\Lambda)^{\whh[\nu] ^{+}}_{kQ}$.
\end{thm}

\section{Character formulae and combinatorial identities}

In this section, we use our main results, Theorems \ref{t1}, \ref{t2}, \ref{t3} and \ref{t4} to compute the  character formulae of the corresponding bases. Furthermore, as an application, in the case of the principal subspace of the generalized Verma module we obtain two new families of combinatorial identities. Throughout the entire section,   $\Lambda$ denotes an arbitrary  rectangular highest weight, i.e. the weight of the form \eqref{weight}.

\subsection{Principal subspaces}
Let $V=N^{\whn}(\Lambda),L^{\whn}(\Lambda)$.
We define the character of the principal subspace $W_{V }^T$ by
$$\mathop{\text{ch}}   W_{V }^T= \sum_{m,r_1,\ldots, r_l\geq 0} 
\dim {W_{V}^T}_{(m,r_1,\ldots, r_l)}q^{m}y^{r_1}_{1}\cdots y^{r_l}_{l},
$$
where ${W_{V }^T}_{(m,r_1,\ldots, r_l)}$ is the weight subspace of the Cartan subalgebra $\h_{(0)}\oplus \mathbb{C}c\oplus \mathbb{C}d$ of weight $-m\delta +r_1{\alpha_1} +\cdots + r_l{\alpha_l}$.

For the dual-charge-type of  quasi-particle monomial as in  \eqref{kv4}  let $p_{i}^{(s)}$ be the number of quasi-particles of color $i$ and  charge $s=1,\ldots ,t$ which appear in the monomial.  Furthermore, we organize the numbers $p_{i}^{(s)}$ into the sequence $\Pc=(\Pc_1, \ldots, \Pc_l)$ such that $\Pc_i =(p_i^{(1)},\ldots ,p_{i}^{(t)})$. Also, we write
$$(a;q)_r=\prod_{i=1}^r (1- aq^{i-1}) \quad \text{for} \quad r\geq  0
\qquad \text{and} \qquad 
(a; q)_{\infty} =\prod_{i\geq  1} (1- aq^{i-1}).
$$
In addition, to express the character formulae we shall need the following notation 
$$
\frac{1}{(q^{\mu_i}; q^{\mu_i})_r}= \sum_{s \geq 0}p_r(s)q^{\mu_is},
$$
where $p_r(s)$ is  the number of partitions of $s$ with most $r$ parts.
By arguing as in \cite[Sect.5]{G}  and using the difference    conditions  \eqref{rel8} and \eqref{rel9}, which  determine the quasi-particle basis of the principal subspace of standard module via Theorem \ref{t1}, we find
\begin{thm}\label{CPSL}
For the affine Lie algebras $A_{2l-1}^{(2)}$ and $D_{l+1}^{(2)}$ we have
$$\mathop{\text{ch}}     W_{L^{\whn}(\Lambda) }^T=\sum_{\Pc}
\frac{q^{\frac{1}{2} \sum_{i,r=1}^{l}\sum_{m,n=1}^k \langle{\alpha_i}_{(0)}, {\alpha_{r}}_{(0)}\rangle \mathrm{min}\{m,n\} p_i^{(m)}p_r^{(n)}+\tilde{p}_j}}
{\prod_{i=1}^{l}\prod_{s=1}^{k}(q^{\mu_i};q^{\mu_i})_{p_i^{(s)}}}\prod_{i=1}^{l}y^{n_i}_{i},$$
where $\tilde{p}_j=\sum_{s=k_0+1}^k (s-k_0) \mu_jp_j^{(s)}$ and $n_i=\sum_{s=1}^k s p_i^{(s)}$. 
The sum goes over all finite sequences $\Pc=(\Pc_1, \ldots, \Pc_l)$ of $lk$ nonnegative integers.
\end{thm}

The character formula for the principal subspace of the generalized Verma module is obtained analogously:
\begin{thm}\label{CPSN}
For the affine Lie algebras $A_{2l-1}^{(2)}$ and $D_{l+1}^{(2)}$ we have
$$\mathop{\text{ch}}    W_{N^{\whn}(\Lambda) }^T=\sum_{\Pc}
\frac{q^{\frac{1}{2} \sum_{i,r=1}^{l}\sum_{m,n=1}^t \langle{\alpha_i}_{(0)}, {\alpha_{r}}_{(0)}\rangle \mathrm{min}\{m,n\} p_i^{(m)}p_r^{(n)}}}
{\prod_{i=1}^{l}\prod_{s=1}^{t}(q^{\mu_i};q^{\mu_i})_{p_i^{(s)}}}\prod_{i=1}^{l}y^{n_i}_{i},$$
where $n_i=\sum_{s=1}^t sp_i^{(s)}$, and the sum goes over all finite sequences $\Pc=(\Pc_1, \ldots, \Pc_l)$ of nonnegative integers  with finite support.
\end{thm}

Due to the vector space isomorphism \eqref{zakarakter}, we can    compute the character of $W_{N^{\whn}(\Lambda)}^T$ from the Poincar\'{e}--Birkhoff--Witt basis of $U(\ntplusleq)$ as well, thus getting
\beq\label{karakter1}
\mathop{\text{ch}}     W_{N^{\whn}(\Lambda)}^T = \frac{1}{\prod_{\alpha \in R_+}(\alpha; q^{ \langle \alpha_{(0)},\alpha_{(0)}\rangle/2  })_{\infty} }.
\eeq
On the right-hand side we used the notation
$$
(\alpha; q)_{\infty}=(qy_1^{a_1}y_2^{a_2}\ldots y_l^{a_l};q)_{\infty},
$$
for any   positive root $\alpha\in R_+$ such that $\alpha_{(0)} =a_1{\alpha_1}_{(0)}+\cdots a_l{\alpha_l}_{(0)}$.
Theorem \ref{CPSN} and the character formula \eqref{karakter1} imply the following generalization of the identities from \cite{Bu2}:
\begin{cor}\label{thm_identitet} 
For the affine Lie algebras $A_{2l-1}^{(2)}$ and $D_{l+1}^{(2)}$ we have
$$
   \displaystyle\frac{1}{\prod_{\alpha \in R_+}(\alpha; q^{ \langle \alpha_{(0)},\alpha_{(0)}\rangle/2  })_{\infty} } = \displaystyle\sum_{\Pc}
\frac{q^{\frac{1}{2} \sum_{i,r=1}^{l}\sum_{m,n=1}^t \langle{\alpha_i}_{(0)}, {\alpha_{r}}_{(0)}\rangle \mathrm{min}\{m,n\} p_i^{(m)}p_r^{(n)}}}
{\prod_{i=1}^{l}\prod_{s=1}^{t}(q^{\mu_i};q^{\mu_i})_{p_i^{(s)}}}\prod_{i=1}^{l}y^{n_i}_{i},
$$
where $n_i=\sum_{s=1}^t sp_i^{(s)}$, and the sum goes over all finite sequences $\Pc=(\Pc_1, \ldots, \Pc_l)$ of nonnegative integers  with finite support.
\end{cor}

\subsection{Standard modules}

Let $\Rc=(r_l^{(1)},\ldots ,r_{1}^{(k-1)})$ be the dual-charge-type of a   quasi-particle monomial $b\in\mathcal{B}_{L^{\whn}(\Lambda)}$.  Consider the  sequences  $\Pc_i =(p_i^{(1)},\ldots ,p_{i}^{(k-1)})$ given by
\beq\label{recall678}
\Pc_i=(r_{i}^{(1)}-r_{i}^{(2)},r_{i}^{(2)}-r_{i}^{(3)},\ldots,r_{i}^{(k-2)}-r_{i}^{(k-1)},r_{i}^{(k-1)}) \quad\text{for}\quad i=1,\ldots ,l.
\eeq
Note that    the number  $p_{i}^{(m)}$ equals  the number of quasi-particles of color $i$ and  charge $m$ in the quasi-particle monomial $b$. We define the character of $L^{\whn}(\Lambda)$ by
$$\mathop{\text{ch}}   L^{\whn}(\Lambda)= \sum_{m,r_1,\ldots, r_l\geq 0} 
\dim {L^{\whn}(\Lambda)}_{(m,r_1,\ldots, r_l)}q^{m}y^{r_1}_{1}\cdots y^{r_l}_{l},$$
where ${L^{\whn}(\Lambda)}_{(m,r_1,\ldots, r_l)}$ is the weight subspace of weight $-m\delta +r_1{\alpha_1} +\cdots + r_l{\alpha_l}$, with respect to the Cartan subalgebra $\h_{(0)}\oplus \mathbb{C}c\oplus \mathbb{C}d$. In a similar way, we can define the character of the vacuum space $L^{\whn}(\Lambda)^{\whh[\nu] ^{+}}$. By the isomorphism in \eqref{LW}  we have
$$
\mathop{\text{ch}}    L^{\whn}(\Lambda)  = \frac{1} {\prod_{i=1}^l(q^{\mu_i }; q^{\mu_i })_{\infty} }\mathrm{ch} \  L^{\whn}(\Lambda)^{\whh[\nu] ^{+}}.
$$
Hence Theorems \ref{t2} and \ref{t3} imply
\begin{thm}
For the affine Lie algebras $A_{2l-1}^{(2)}$ and $D_{l+1}^{(2)}$ we have
\begin{align*}
\mathop{\text{ch}}    L^{\whn}(\Lambda) =&\,\frac{1}{\prod_{i=1}^{l}(q^{\mu_i};q^{\mu_i})_{\infty}}\sum_{\Pc}
\frac{q^{\frac{1}{2} \sum_{i,r=1}^{l}\sum_{m,n=1}^{k-1} \langle{\alpha_i}_{(0)}, {\alpha_{r}}_{(0)}\rangle \mathrm{min}\{m,n\} p_i^{(m)}p_r^{(n)}+\tilde{p}_j}}
{\prod_{i=1}^{l}\prod_{s=1}^{k-1}(q^{\mu_i};q^{\mu_i})_{p_i^{(s)}}} 
\\
&\times \sum_{\alpha \in Q} q^{\bigl \langle \alpha_{(0)}, \frac{k}{2} \alpha_{(0)} + \Lambda +\sum_{s=1}^{k-1}sp_i^{(s)}{\alpha_{i}}_{(0)}\bigr \rangle}\prod_{i=1}^{l}y^{h_{\lambda_i}(k\alpha_{(0)}+\Lambda+ \sum_{s=1}^{k-1}sp_i^{(s)}{\alpha_{i}}_{(0)})}_{i},
\end{align*}
where the first sum   goes over all finite sequences $\Pc=(\Pc_1, \ldots, \Pc_l)$ of $l(k-1)$ nonnegative integers, $\tilde{p}_j=\sum_{s=k_0+1}^{k-1} (s-k_0) \mu_jp_j^{(s)}$ and $h_{\lambda_i}$ are  the fundamental coweights of $\mathfrak{g}_{(0)}$.
\end{thm}

As an illustration of the theorem, we now  give some examples of  character formulae. 

\begin{ex}
For the affine Lie algebra
 $A_5^{(2)}$ we have
\begin{align*}
\mathop{\text{ch}}      L^{\whn}(\Lambda_1) =&\, \frac{1}{\prod_{i=1}^{3}(q^{\mu_i};q^{\mu_i})_{\infty}} \sum_{\alpha \in Q} q^{ \langle \alpha_{(0)}, \frac{1}{2}\alpha_{(0)} +\Lambda_1\rangle}\prod_{i=1}^{3}y^{h_{\lambda_i}(\alpha_{(0)}+\Lambda_1)}_{i},\\
\mathop{\text{ch}}      L^{\whn}(2\Lambda_1) =&\,\frac{1}{\prod_{i=1}^{3}(q^{\mu_i};q^{\mu_i})_{\infty}} \sum_{(\Pc_1,\Pc_2,\Pc_3 )}
\frac{q^{\frac{1}{2}  ({p_1^{(1)}}^2-p_1^{(1)}p_2^{(1)}+{p_2^{(1)}}^2-2p_2^{(1)}p_3^{(1)}+2{p_3^{(1)}}^2)+\frac{1}{2}p_1^{(1)}}}
{\prod_{i=1}^{3}(q^{\mu_i};q^{\mu_i})_{p_i^{(1)}}} \\
& \times \sum_{\alpha \in Q}q^{\langle \alpha_{(0)}, \alpha_{(0)}+2\Lambda_1 +\sum_{i=1}^3p_i^{(1)}{\alpha_i}_{(0)}\rangle}\prod_{i=1}^{3}y^{h_{\lambda_i}(2\alpha_{(0)}+2\Lambda_1+\sum_{i=1}^3p_i^{(1)}{\alpha_i}_{(0)})}_{i},\\
\mathop{\text{ch}}      L^{\whn}(\Lambda_0+\Lambda_1) =&\,\frac{1}{\prod_{i=1}^{3}(q^{\mu_i};q^{\mu_i})_{\infty}} \sum_{(\Pc_1,\Pc_2,\Pc_3 )}
\frac{q^{\frac{1}{2}  ({p_1^{(1)}}^2-p_1^{(1)}p_2^{(1)}+{p_2^{(1)}}^2-2p_2^{(1)}p_3^{(1)}+2{p_3^{(1)}}^2)}}
{\prod_{i=1}^{3}(q^{\mu_i};q^{\mu_i})_{p_i^{(1)}}} \\
& \times \sum_{\alpha \in Q}q^{\langle \alpha_{(0)}, \alpha_{(0)}+\Lambda_0+\Lambda_1 +\sum_{i=1}^3p_i^{(1)}{\alpha_i}_{(0)}\rangle}\prod_{i=1}^{3}y^{h_{\lambda_i}(2\alpha_{(0)}+\Lambda_0+\Lambda_1+\sum_{i=1}^3p_i^{(1)}{\alpha_i}_{(0)})}_{i}.
\end{align*}
\end{ex}

\begin{ex}
For the affine Lie algebra
 $D_3^{(2)}$ we have
\begin{align*}
\mathop{\text{ch}}      L^{\whn}(\Lambda_2) =&\, \frac{1}{\prod_{i=1}^{2}(q^{\mu_i};q^{\mu_i})_{\infty}} \sum_{\alpha \in Q} q^{ \langle \alpha_{(0)}, \frac{1}{2}\alpha_{(0)} +\Lambda_2\rangle}\prod_{i=1}^{2}y^{h_{\lambda_i}(\alpha_{(0)}+\Lambda_2)}_{i},\\
\mathop{\text{ch}}   L^{\whn}(2\Lambda_2) =&\,\frac{1}{\prod_{i=1}^{2}(q^{\mu_i};q^{\mu_i})_{\infty}} \sum_{(\Pc_1,\Pc_2)}
\frac{q^{\frac{1}{2}  ({2p_1^{(1)}}^2-2p_1^{(1)}p_2^{(1)}+{p_2^{(1)}}^2)+\frac{1}{2}p_2^{(1)}}}
{\prod_{i=1}^{2}(q^{\mu_i};q^{\mu_i})_{p_i^{(1)}}} \\
  &\times \sum_{\alpha \in Q}q^{\langle \alpha_{(0)}, \alpha_{(0)}+2\Lambda_2 +\sum_{i=1}^2p_i^{(1)}{\alpha_{i}}_{(0)}\rangle}\prod_{i=1}^{2}y^{h_{\lambda_i}(2\alpha_{(0)}+2\Lambda_2+\sum_{i=1}^2p_i^{(1)}{\alpha_{i}}_{(0)})}_{i},\\
	\mathop{\text{ch}}    L^{\whn}(\Lambda_0+\Lambda_2) =&\,\frac{1}{\prod_{i=1}^{2}(q^{\mu_i};q^{\mu_i})_{\infty}} \sum_{(\Pc_1,\Pc_2)}
\frac{q^{\frac{1}{2}  ({2p_1^{(1)}}^2-2p_1^{(1)}p_2^{(1)}+{p_2^{(1)}}^2)}}
{\prod_{i=1}^{2}(q^{\mu_i};q^{\mu_i})_{p_i^{(1)}}} \\
& \times \sum_{\alpha \in Q}q^{\langle \alpha_{(0)}, \alpha_{(0)}+\Lambda_0+\Lambda_2 +\sum_{i=1}^2p_i^{(1)}{\alpha_{i}}_{(0)}\rangle}\prod_{i=1}^{2}y^{h_{\lambda_i}(2\alpha_{(0)}+\Lambda_0+\Lambda_2+\sum_{i=1}^2p_i^{(1)}{\alpha_{i}}_{(0)})}_{i}.
\end{align*}
\end{ex}

\subsection{Parafermionic spaces}

As in \cite{OT}, to define the character of the parafermionic space we use the parafermionic grading operator $D$   on  $L^{\whn} (\Lambda)^{\whh[\nu] ^{+}}_{kQ}$ given by
\beq\label{degop}
D=-d-D^{\whh[\nu] ^{+}}
\quad\text{and}\quad
D^{\whh[\nu] ^{+}}\bigg|\bigg._{L^{\whn} (\Lambda)^{\whh[\nu] ^{+}}_\mu}= \textstyle\frac{\langle \mu_{(0)}, \mu_{(0)}\rangle}{2k}-\frac{\langle \Lambda, \Lambda\rangle}{2k}.
\eeq
Note that by \eqref{degop}    the conformal energy of the basis vector
\beq\label{basisv3}
  \psi_{\Rc'}(\Ec)  v_{  \Lambda }
	=  \psi_{(n_{r_l^{(1)},l},\ldots ,n_{1,1})}(m_{r_l^{(1)},l},\ldots ,m_{1,1})  v_{\Lambda },
\eeq
which corresponds to the quasi-particle monomial 
$$
 \,x^{\whn}_{n_{r_{l}^{(1)},l}\alpha_{l}}(m_{r_{l}^{(1)},l}) \ldots  x^{\whn}_{n_{1,l}\alpha_{l}}(m_{1,l})\ldots 
x^{\whn}_{n_{r_{1}^{(1)},1}\alpha_{1}}(m_{r_{1}^{(1)},1}) \ldots  x^{\whn}_{n_{1,1}\alpha_{1}}(m_{1,1}) 
,
$$
is equal to 
\begin{align*}
&\textstyle-\sum_{i=1}^l\sum_{u=1}^{r_i^{(1)}}m_{u,i}
-\frac{k_j}{k} \sum_{t=1}^{k -1} t p_j^{(t)} \\
&-\textstyle\sum_{i=1}^l\sum_{u=1}^{r_i^{(1)}} 
\left(
\frac{n_{u,i}^2\mu_i}{k}+\frac{1}{k}\textstyle
\left<
n_{u,i}{\alpha_i}_{(0)},
\sum_{s=1}^{u-1} n_{s,i}{\alpha_i}_{(0)}+\sum_{p=1}^{i-1}\sum_{s=1}^{r_p^{(1)}}n_{s,p}{\alpha_p}_{(0)}
\right>
\right), 
\end{align*}
where the second summand is due to the identity
$$
\textstyle
-\frac{1}{k}\textstyle
\left<\textstyle \sum_{t=1}^{k -1} t p_j^{(t)}{\alpha_j}_{(0)},\Lambda \right>
=-\frac{k_j}{k}\mu_j \textstyle\sum_{t=1}^{k -1} t p_j^{(t)}.
$$

Define  the character of the parafermionic space $L^{\whn} (\Lambda)^{\whh[\nu] ^{+}}_{kQ}$ by
$$\mathop{\text{ch}}     L^{\whn} (\Lambda)^{\whh[\nu] ^{+}}_{kQ}=\sum_{m, r_1, \ldots, r_l} \dim  {L^{\whn} (\Lambda)^{\whh[\nu] ^{+}}_{kQ}}_{(m, r_1, \ldots, r_l)}q^m,$$
where ${L^{\whn} (\Lambda)^{\whh[\nu] ^{+}}_{kQ}}_{(m, r_1, \ldots, r_l)}$ is the weight space spanned by monomial vectors (\ref{basisv3}) of conformal energy $-m$ and color-type $( r_1, \ldots, r_l)$. 

Recall the notation from \eqref{recall678} 
and introduce the following expressions:
\begin{gather*}
  D_{\Pc}(q)= \frac{1}{\prod_{i=1}^{l}\prod_{s=1}^{k-1}(q^{\mu_i};q^{\mu_i})_{p_i^{(s)}}},\qquad B_{\Pc}(q)=q^{\mu_j\sum_{t=k_0+1}^{k-1}(t-k_0)p_j^{(t)}}  q^{ -\mu_j\frac{k_j}{k} \sum_{t=1}^{k -1} t p_j^{(t)} },\\
G_{\Pc}(q) = q^{\frac{1}{2} \sum_{i,r=1}^{l}\sum_{m,n=1}^{k-1}\langle{\alpha_i}_{(0)}, {\alpha_{r}}_{(0)}\rangle \left(\mathrm{min}\{m,n\}-\frac{mn}{k}\right) p_i^{(m)} p_r^{(n)}}  .
\end{gather*}
From the above considerations and Theorem \ref{t4} we get 
\begin{thm}\label{xt4}
For the affine Lie algebras $A_{2l-1}^{(2)}$ and $D_{l+1}^{(2)}$ we have
\beq \label{fin} \mathop{\text{ch}}     L^{\whn} (\Lambda)^{\whh[\nu] ^{+}}_{kQ}=
\sum_{\Pc}D_{\Pc}(q)\ts G_{\Pc}(q)\ts B_{\Pc}(q),
\eeq
where the sum goes over all finite sequences $\Pc=(\Pc_l,\ldots ,\Pc_1)$ of $l(k-1)$ nonnegative integers such that  $\Pc_i =(p_i^{(1)},\ldots ,p_{i}^{(k-1)})$. 
\end{thm}
It is worth noting that the   character formula  \eqref{fin} coincides with    \cite[Eq. (14)]{Gep2}.
At the end, let us
   give some examples of the parafermionic character formulae.
	
	\begin{ex}
	For the affine Lie algebra  $A_5^{(2)}$  we have
\begin{align*}
\mathop{\text{ch}}   L^{\whn}(2\Lambda_1)^{\whh[\nu] ^{+}}_{2Q} =&\,\sum_{\Pc=(\Pc_1,\Pc_2,\Pc_3)}
\frac{q^{\frac{1}{4}  ({p_1^{(1)}}^2-p_1^{(1)}p_2^{(1)}+{p_2^{(1)}}^2-2p_2^{(1)}p_3^{(1)}+2{p_3^{(1)}}^2)}}
{\prod_{i=1}^{3}(q^{\mu_i};q^{\mu_i})_{p_i^{(1)}}}, \\
\mathop{\text{ch}} L^{\whn}(\Lambda_0+\Lambda_1)^{\whh[\nu] ^{+}}_{2Q}  =&\, \sum_{\Pc=(\Pc_1,\Pc_2,\Pc_3)}
\frac{q^{\frac{1}{4}  ({p_1^{(1)}}^2-p_1^{(1)}p_2^{(1)}+{p_2^{(1)}}^2-2p_2^{(1)}p_3^{(1)}+2{p_3^{(1)}}^2)-\frac{1}{4}p_1^{(1)}}}
{\prod_{i=1}^{3}(q^{\mu_i};q^{\mu_i})_{p_i^{(1)}}} .
\end{align*}
	\end{ex}

\begin{ex}
For the affine Lie algebra
 $D_3^{(2)}$ we have
\begin{align*}
\mathop{\text{ch}}    L^{\whn}(2\Lambda_2)^{\whh[\nu] ^{+}}_{2Q} =&\, \sum_{\Pc=(\Pc_1,\Pc_2)}
\frac{q^{\frac{1}{4}  ({2p_1^{(1)}}^2-2p_1^{(1)}p_2^{(1)}+{p_2^{(1)}}^2)}}
{\prod_{i=1}^{2}(q^{\mu_i};q^{\mu_i})_{p_i^{(1)}}}, \\
\mathop{\text{ch}}    L^{\whn}(\Lambda_0+\Lambda_2)^{\whh[\nu] ^{+}}_{2Q}  =&\, \sum_{\Pc=(\Pc_1,\Pc_2)}
\frac{q^{\frac{1}{4}  ({2p_1^{(1)}}^2-2p_1^{(1)}p_2^{(1)}+{p_2^{(1)}}^2)-\frac{1}{4}p_2^{(1)}}}
{\prod_{i=1}^{2}(q^{\mu_i};q^{\mu_i})_{p_i^{(1)}}} .
\end{align*}
\end{ex}

\section*{Acknowledgement}
The authors are grateful to Mirko Primc for numerous helpful discussions. The authors would also like to thank Ole Warnaar
for bringing to their attention some useful references on combinatorial identities.
This work has been supported in part by Croatian Science Foundation under the project UIP-2019-04-8488.
The first author is partially supported by the QuantiXLie Centre of Excellence, a project cofinanced by the Croatian Government and European Union through the European Regional Development Fund - the Competitiveness and Cohesion Operational Programme (Grant KK.01.1.1.01.0004). 

\linespread{1.0}

\end{document}